\documentclass[12pt, reqno, a4paper] {amsart}

\pdfoutput=1

\usepackage{graphicx}
\usepackage{tikz-cd}
\usepackage{amssymb}
\usepackage{hyperref}
\usepackage{enumerate}
\usepackage[nocompress]{cite}
\usepackage{bbm}
\usepackage[margin=3.14cm]{geometry}
\usepackage{adjustbox}
\usepackage{amsaddr}
\usepackage{microtype}

\newtheorem{theorem}{\bf Theorem}[section]
\newtheorem{proposition}[theorem]{\bf Proposition}
\newtheorem{lemma}[theorem]{\bf Lemma}
\newtheorem{corollary}[theorem]{\bf Corollary}

\theoremstyle{definition}
\newtheorem{definition}[theorem]{\bf Definition}
\newtheorem*{definition*}{\bf Definition}
\newtheorem{remark}[theorem]{\bf Remark}
\newtheorem{example}[theorem]{\bf Example}

\DeclareTextFontCommand{\emph}{\bfseries\em}

\newcommand		{\p}[1]	{\left(#1\right)}

\newcommand		{\ppp}[1]{\left\{#1\right\}}
\newcommand		{\abs}[1]{\left|#1\right|}

\newcommand 	{\norm}[1]{\left\lVert#1\right\rVert}
\newcommand 	{\enorm}{\norm{\ \ }}
\newcommand 	{\sset}[1]{\left\{#1\right\}}

\newcommand		{\R} 	{\mathbb R}
\newcommand		{\Z} 	{\mathbb Z}
\newcommand		{\N} 	{\mathbb N}

\newcommand		{\I} 	{I}

\newcommand 	{\flow}	{\Phi}
\newcommand		{\D}	{\mathcal D}

\newcommand 	{\Dih} {\mathrm D}
\newcommand     {\Aut}{\mathrm{Aut}}

\newcommand 	{\Iso} {\mathrm {Iso}}
\renewcommand 	{\mp} {\mathfrak {Mp}}
\newcommand 	{\fix} {\mathrm {Fix}}
\newcommand 	{\cluster} {C}

\renewcommand 	{\c}	{\bullet}
\renewcommand 	{\d}	{\mathrm d}

\renewcommand 	{\mod} {\mathrm {mod} \,}

\newcommand 	{\1} {\mathbbm{1}}
\newcommand 	{\ER}	{Erd\H{o}s-Renyi\ }
\newcommand{\at}{\makeatletter @\makeatother}

\newcommand		{\T} 	{\mathbb T}
\renewcommand 	{\S} 	{\mathbb S}

\newcommand 	{\davide}[1]{{#1}}

\begin{document}

\title[Dynamical Systems on Graph Limits and their Symmetries]
{Dynamical Systems on Graph Limits \\and their Symmetries}
\author[Christian Bick]{Christian Bick$^{1,2,*}$, Davide Sclosa$^{1,\star}$}
\address{$^{1}$Department of Mathematics, Vrije Universiteit Amsterdam, De Boelelaan 1111, 1081 HV Amsterdam, the Netherlands\\
$^2$ Institute for Advanced Study, Technical University of Munich, Lichtenbergstra\ss{}e 2, 85748 Garching, Germany\\
e-mail: $^*$c.bick\at vu.nl; $^{\star}$d.sclosa\at vu.nl}

\begin{abstract}
The collective dynamics of interacting dynamical units on a network crucially depends on the properties of the network structure. 
Rather than considering large but finite graphs to capture the network, one often resorts to graph limits and the dynamics thereon.
We elucidate the symmetry properties of dynamical systems on graph limits---including graphons and graphops---and analyze how the symmetry shape\davide{s} the dynamics, for example through invariant subspaces.
In addition to traditional symmetries, dynamics on graph limits
can support generalized noninvertible symmetries.
Moreover, as asymmetric networks can have symmetric limits, we note that one can expect to see ghosts of symmetries in the dynamics of large but finite asymmetric networks.
\end{abstract}

\maketitle


\section{Introduction} \label{sec:introduction}

Synchronization and other collective phenomena of networks of interconnected dynamical systems are crucial in many systems in science and technology, ranging from interacting neural units to power grid networks~\cite{Pikovsky2003,Strogatz2004}.
The network structure is often captured by a graph:
Each vertex is a dynamical system and two systems are coupled if they are
connected by an edge.
Thus, an essential question in network dynamical systems is how the network structure shapes the collective network dynamics.
Classical dynamical systems tools are of limited use to answer this question as many relevant network dynamical systems have many nodes---e.g., the human brain is a network of billions of individual cells. 
Thus, a common approach is to consider continuum limits, such as graphons~\cite{lovasz2012large} or graphops~\cite{backhausz2022action}, and dynamics on these limit objects
(see, e.g.,~\cite{Medvedev2013a, kaliuzhnyi2018mean,
Chiba2016, medvedev2014small, Chiba2019,Kuehn2020b,gkogkas2022graphop}).

In this paper we analyze the symmetry properties of graph limits and the implications for the network dynamics.
Symmetries essentially shape the dynamics, for example, by inducing invariant subspaces~\cite{Golubitsky2002} that may correspond to synchrony patterns~\cite{Pecora2014}, or facilitate the emergence of structurally stable heteroclinic cycles~\cite{Ashwin2017}.
{First, we make the notion of symmetries rigorous for dynamics on graph limits; 
this leads to some technical challenges compared to the finite-dimensional setting as the dynamics have to be defined on appropriate spaces where the state of
individual vertices might become irrelevant.}
Second, we compute the symmetry groups of several graphons and graphops appearing in \davide{the} literature, including homogeneous networks, multiple populations, and networks on manifolds.
Third, we describe the corresponding effects on dynamics.
Fourth, we give some consequences for large but finite-dimensional systems: As the limit object can have significantly more symmetries than any finite element of the converging sequence one can expect ``ghosts'' of symmetries of the limit object in large but finite network dynamical systems.

\subsection*{Networks, graphs, and graph limits}
A graph provides a natural mathematical abstraction of a network:
Two vertices interact if there is an edge between them.
Since many networks of interest have many vertices---e.g., the human brain mentioned above or the internet as a network of webpages connected by links---it is often convenient to consider the limit of large graphs.
A graph sequence is called dense if the number of edges grows quadratically with respect to the number of vertices.
The groundbreaking work by Szegedy and Lov\'asz~\cite{lovasz2012large} introduces graphons as limits of dense graph sequences.
Graphops have been later introduced to include both dense and sparse graph sequences~\cite{backhausz2022action}.

Graph automorphisms describe the symmetries of a graph: 
These are permutations of vertices that preserve the edges.
Here we typically assume graphs to be finite, undirected, and simple and make it explicit when additional structure (e.g., weights) is present.
In this context, the automorphism group of a cycle graph is a dihedral group and of a complete graph is the whole permutation group.
While the notion of an automorphism can be extended to graphons as graph limits, there are some nuances that need to be taken into account~\cite{lovasz2015automorphism}. 
By contrast, automorphisms---and thus the symmetries---of the more general class of graphops have rarely been considered.

\subsection*{Dynamical systems and symmetries}
Symmetries of a dynamical system are transformations of phase space that send trajectories to trajectories; see~\cite{Golubitsky2002,Golubitsky1988} for an introduction to the subject.
To fix some notation, let the flows~$(X, \flow^X)$ and~$(Y, \flow^Y)$ define dynamical systems. 
A map~$\gamma: X\to Y$ \emph{maps trajectories to trajectories} if for every~$x\in X$ and every~$t\in \R$ we have~$\gamma(\flow^X_t (x)) = \flow^Y_t(\gamma(x))$. A \emph{symmetry}~$\gamma$ of a dynamical system~$(X, \flow)$ is a bijection~$\gamma:X\to X$ sending trajectories to trajectories: 
For every~$t\in \R$ and every~$x\in X$ we have~$\gamma(\Phi_t(x)) = \Phi_t (\gamma(x))$.
Note that being a symmetry is the same as commuting with the flow.
Notice that if~$\gamma$ is a symmetry then its inverse~$\gamma^{-1}$ is also a symmetry since
$\gamma^{-1} \circ \flow_t = (\flow_{-t} \circ \gamma)^{-1} = \flow_t \circ \gamma^{-1}$
and so the symmetries form a group. 
Typically one is interested in group of symmetries of a specific type, for example, due to the physical nature of the system.
Symmetries give rise to dynamically invariant subspaces.
It is a general fact that if two maps~$\gamma, \nu: X\to X$ commute with each other, then the fixed point set~$\fix(\gamma)$ is $\nu$-invariant and the fixed point set~$\fix(\nu)$ is $\gamma$-invariant.
If~$\gamma$ is a symmetry, this implies that the fixed point set~$\fix (\gamma)$ is dynamically invariant.

In the case of coupled dynamical systems on a graph, symmetries may be induced by the underlying combinatorial graph structure independently of the specific choice of coupling functions.
Importantly, symmetry-induced invariant subspaces may correspond to synchrony patterns, where individual units have the same state. 
Thus, computing symmetry groups allows
to understand the emergence of ``cluster'' dynamics~\cite{Pecora2014}. 
Moreover, there has been further interest in symmetries of network dynamical systems due to symmetry breaking phenomena~\cite{Bick2015d,Kemeth2018}.
However, symmetry considerations for network dynamical systems typically focus on finite-dimensional systems.
Elucidating the symmetries of network dynamical systems on graph limits as well as the implications for the dynamics remain underexplored.

\subsection*{Main contributions}


In this paper, we analyze the symmetry properties of graph limits, their implications for dynamics on such graph limits, and corresponding large- but finite-dimensional network dynamical systems.

From the perspective of network structure alone, we focus on the automorphism group of graphons that are common in literature. 
These include constant graphons (Section~\ref{sec:constant_graphon}), canonical embeddings of graphs on the unit interval (Proposition~\ref{prop:block_structure_group}), and a class of graphons determined by distant-dependent coupling on manifolds (Theorem~\ref{thm:iso_is_auto}). 
For the latter, we highlight that the manifold itself is the natural choice to index the vertices rather than the traditional choice of indexing vertices by the unit interval.
Specific examples we analyze are the spherical graphon (Section~\ref{sec:spherical_graphon}) and graphons on the torus (Section~\ref{sec:toral_graphon}).
We also introduce a notion of graphop automorphisms and analyze the automorphism group of the spherical graphop (Proposition~\ref{prop:spherical_graphop}).

Automorphisms of the network structure induce symmetries for dynamics on the network.
We generalize this observation from graphs to a general setting for dynamical systems on graphons (Corollary~\ref{cor:imp_gives_isometry}) and graphops (Lemma~\ref{lem:graphop_symmetry}).
{
As the dynamics on graph limits are typically infinite-dimensional, we consider the dynamics on graph limits as dynamics on~$L^1(J)$, where~$J$ is an appropriate index space.
First, the general setup induces technical challenges: state of any given vertex (or pair of vertices) that define synchrony for dynamics on graphs becomes meaningless for typical dynamics on graph limits.
}%
Second, compared to their finite-dimensional counterparts, dynamical systems on graph limits may have generalized symmetries that can be noninvertible.
This is due to the existence of non-invertible measure preserving transformations, like the doubling map~$x \mapsto 2x\ (\mod 1)$ on the unit interval.
{
Note that graphons that are equivalent as graph limits yield distinct dynamical systems.
We show that while these may have wildly different symmetry groups they can be related to one another (Section~\ref{sec:graph_limit_theory}).
}

We subsequently apply our results to the study of graphon dynamical systems with relevant topologies, with a focus on dynamical phenomena that are due to symmetry.
We show that dynamically invariant subspaces arise in two different ways, as a set of fixed points of a symmetry or as the image of the Koopman operator (Section~\ref{sec:InvSubsp}).
Cluster dynamics and multi-population structures are well-studied in finite networks and generalize to graphons, see Theorem~\ref{thm:twins_constant} and Section~\ref{sec:block_structure}.
Moreover, we analyze systems with spherical symmetry~(Figure~\ref{fig:sphere}), and multi-dimensional twisted states for coupled oscillators on a torus~(Section~\ref{sec:toral_graphon}). 
We also consider mean-field dynamics on graphons where the state of each node is represented by a probability measure: 
Using symmetries we show that dynamical systems on graphs~\cite{Rodrigues2016}, mean-field dynamics~\cite{Strogatz1991,Lancellotti2005}, multi-population mean fields~\cite{bick2022multi}, and dynamics on graphons~\cite{Medvedev2013a} are all dynamically invariant subspaces of the system analyzed in~\cite{kaliuzhnyi2018mean}.

Finally, we return to large- but finite-dimensional dynamical systems and their relationship to dynamics on the graph limit. 
Note that convergent sequences of asymmetric graphs can have symmetric limit: 
For example, \ER random graphs are asymmetric with asymptotic probability~$1$ while they converge to the constant graphon, which has a large symmetry group.
Thus, it is natural to expect that the dynamics on large (but finite) graphs will inherit dynamical features from the symmetric limit.
We make this observation rigorous in Section~\ref{sec:approximated_symmetries} which explains our numerical findings in Figure~\ref{fig:sphere}.

\subsection*{Structure of the paper}

The paper is organized as follows. In Section~\ref{sec:graph_dynamical_system} we set the stage by considering a class of dynamical systems on graphs and their symmetries.
In the following Section~\ref{sec:graphon_dynamical_system}, we generalize this class of dynamical systems to graphons and show that graphon automorphisms induce symmetries on the dynamics on the graphon. In Section~\ref{sec:homogeneous_twins} we exploit these results to analyze symmetries and invariant subspaces of dynamical systems with relevant network topologies, such as Kuramoto-type networks that consist of one (or more) coupled populations. In Section~\ref{sec:geodesic_coupling} we consider networks that arise when the dynamical units are placed on a manifold and coupling depends on the geodesic distance between two units on the manifold; with the geometric structure,
choosing index spaces different from the unit interval becomes relevant. In Section~\ref{sec:approximated_symmetries}, we compare the dynamics of finite systems with those of the infinite-dimensional limit. In Section~\ref{sec:mean-field-graphon} we consider a first generalization to mean-field systems where the state of each units is a probability measure. Symmetry arguments allow to see relationship to other continuum limits.
Finally, in Section~\ref{sec:graphop} we consider symmetries of graphops and dynamical systems on graphops as a second generalization.

\subsection*{Acknowledgments}

The authors would like to thank C.~Kuehn and Y.~Zhao for helpful discussions, the anonymous referees for helpful comments, and the hospitality of the Institute of Advanced Study at the Technische Universit\"at M\"unchen (TUM--IAS), where part of the work was carried out.
CB acknowledges support from the Engineering and Physical Sciences Research Council (EPSRC) through the grant EP/T013613/1 and a Hans Fischer Fellowship from TUM--IAS.


\section{Graph dynamical systems and their symmetries} \label{sec:graph_dynamical_system}
\label{sec:graph_dynamics}

We first consider a class of network dynamical systems whose underlying network structure is determined by a graph. 
We typically consider finite, undirected, simple graphs and will simply refer to them as \emph{graphs}---we will highlight if a graph comes with extra structure (e.g., weights). 
A \emph{graph isomorphism} is a bijective map between the vertices preserving
edges and non-edges.
Intuitively, isomorphic graphs are the same up to relabeling.
This section will set the stage for a generalization we will subsequently discuss: 
The main takeaway is that automorphisms of the underlying graph yield a group of symmetries of the dynamical system.

\subsection{Graph dynamical systems and transformations}

To be concrete, let~$G$ be a graph with vertex set~$\{1,\ldots,n\}$ and let~$(A_{jk})_{j,k}$ denote its adjacency matrix.
Now suppose that the state of vertex~$j$ is~$u_j \in \R$.
For Lipschitz continuous functions~$f,g:\R^2\to\R$ we consider the \emph{graph dynamical system} where the state of node~$j$ evolves according to
\begin{align}\label{eq:main:graph}
	\dot u_j &= f\p{u_j, \frac{1}{n} \sum_{k=1}^n A_{jk} g(u_j, u_k)}, \qquad j=1,\ldots,n.
\end{align}
Inspired by the neural networks literature we call~$f$ the \emph{activation function}
and~$g$ the \emph{coupling function}. 
These function are understood to be fixed while we will change the graph~$G$.

Note that the vector field is Lipschitz continuous in~$u$ and therefore~\eqref{eq:main:graph}
defines a dynamical system on~$\R^n$.
We denote by~$\mathcal D(G)$ the dynamical system and by~$\flow^G_t(u)$ the corresponding \emph{flow}.
\davide{In this paper a dynamical system is the datum of a phase space together a the flow.}

Graph dynamical systems~\eqref{eq:main:graph} include as special cases the linear diffusion
equation, the replicator equation (with equal types), and the Kuramoto model of identical phase oscillators on a graph
\begin{equation} \label{eq:graph:kuramoto}
	\dot u_j = \frac{1}{n} \sum_{k=1}^n A_{jk} \sin(u_k - u_j + \alpha), \qquad j=1,\ldots,n.
\end{equation}
We will come back to~\eqref{eq:graph:kuramoto} to illustrate the results by numerical simulations.

Isomorphic graphs support the same graph dynamical systems:

\begin{theorem} \label{thm:isomorphism:graph}
Let~$\varphi: G\to H$ an isomorphism of graphs.
Then the map
\[
	\varphi^*(v_1, \ldots, v_n) = (v_{\varphi(1)},\ldots,v_{\varphi(n)})
\]
is an isometry of~$\R^n$ mapping trajectories of the graph dynamical system~$\mathcal D(H)$ to trajectories of the graph dynamical system~$\mathcal D(G)$. In particular, the following diagram is commutative:
\begin{center}
\begin{tikzcd}
G \arrow[r, "\varphi"] \arrow[d] & H \arrow[d] \\
\mathcal D (G) & \mathcal D (H) \arrow[l, "\varphi^*"]
\end{tikzcd}.
\end{center}
\end{theorem}
\begin{proof}
Clearly~$\varphi^*$ is an isometry. It remains to prove that~$\varphi^*$ maps solutions to
solutions. Let~$v$ a solution in~$\D(H)$. Every vertex in~$H$ is the $\varphi$-image
of a vertex of~$G$ and
\begin{align*}
	\dot v_{\varphi(j)} &= f \p{v_{\varphi(j)}, \sum_{k\in v(H)}
			g\p{H_{\varphi(j) k}, v_{\varphi(j)}, v_k}} \\
		&= f \p{v_{\varphi(j)}, \sum_{k \in v(G)}
			g\p{H_{\varphi(j) \varphi(k)}, v_{\varphi(j)}, v_{\varphi(k)}}} \\
		&= f \p{v_{\varphi(j)}, \sum_{k \in v(G)}
			g\p{G_{jk}, v_{\varphi(j)}, v_{\varphi(k)}}}.
\end{align*}
We conclude that~$\varphi^* v$ is a solution in~$\D(G)$.
\end{proof}

\begin{corollary}
If the graphs~$G$ and~$H$ are isomorphic, then the graph dynamical systems~$\mathcal D(G)$ and~$\mathcal D(H)$ are topologically conjugated.
\end{corollary}

\begin{remark}
In the equivariant dynamical systems literature, the action is typically defined as~$\varphi^* (u)_j = u_{\varphi^{-1}(j)}$.
We prefer not to do so, as we will also consider noninvertible maps~$\varphi$ in the following section.
\end{remark}

Theorem~\ref{thm:isomorphism:graph} shows that graph isomorphisms map trajectories
to trajectories.
Note that, in general, a graph homomorphism does not map trajectories to trajectories.
For example, consider~\eqref{eq:graph:kuramoto} on
a connected graph~$G$ with more than one vertex, let~$v$ be a vertex
and~$H$ be the graph with one vertex~$v$ and no edges.
Then~$v$ has trivial dynamics in~$H$ and non-trivial dynamics in~$G$.
The embedding~$H\to G$ does not correspond to a map between dynamical systems.
In~\cite{deville2015modular} the notion of graph dynamical system is different,
but a similar analysis of these functorial aspects appear.

\subsection{Symmetries of graph dynamical systems and invariant subspaces} 
\label{sec:graph_dynamics}

An immediate consequence of Theorem~\ref{thm:isomorphism:graph} is that graph automorphisms yield symmetries of the corresponding graph dynamical systems. This motivates the following definition:

\begin{definition}
Let~$G$ be a graph with~$n$ vertices. Let~$\varphi: G\to G$ an automorphism of the graph.
The map~$\varphi^*: \R^n \to \R^n$ is an isometry of~$\R^n$ and a symmetry of the dynamical system~$\D(G)$.
We call~$\varphi^*$ a \emph{graph-induced symmetry}.
\end{definition}

Note that when we talk about `graph automorphisms' we refer to the symmetries of a graph as a combinatorial object. By contrast, the `graph-induced symmetries' are symmetries of dynamical systems induced by the properties of the underlying graph.
The automorphism group of~$G$
and the group of graph-induced symmetries of~$\D(G)$
are, by definition, isomorphic. We make a clear distinction
as they act on different spaces, the set~$\{1,\ldots,n\}$ and~$\R^n$, respectively.

We remark that the graph-induced symmetry group can be arbitrarily complicated. Indeed,
it is known that any finite group is the automorphism group
of a graph~\cite{frucht1939herstellung}.

As an example, the role of graph automorphisms for the dynamics of symmetrically coupled phase oscillators~\eqref{eq:graph:kuramoto} has been analyzed in~\cite{ashwin1992dynamics}.

Graph dynamical system may have more symmetries than the graph-induced symmetries.
For example, for every dynamical system~$(X, \flow)$
and every~$t\in \R$ the map~$\flow_t$ is a symmetry.
In the context of~\eqref{eq:main:graph}, extra symmetries can appear
for particular choices of~$f$ and~$g$.
If~$f=0$ dynamics is trivial and any bijection of~$X$ is a symmetry.
If~$g(u,v)=\sin(u-v)$ the phase shift group of
rotational symmetries appear~\cite{ashwin1992dynamics}.
Usually one is interested in symmetries within a specific group,
for example linear transformations, homeomorphisms, or isometries. 
In this paper we focus on the symmetries given by the combinatorial structure,
which are independent on the particular choice of~$f$ and~$g$.

Symmetries induce dynamically invariant subspaces. In the case of graph-induced symmetries,
the subspaces are linear and given by equalities of
coordinates---sets of this form are also called \emph{cluster} or \emph{polydiagonal subspaces}.
Let~$\varphi$ be a graph automorphism. We have seen that~$\varphi^*$ is a symmetry,
and therefore the subspace
\begin{equation} \label{eq:graph_cluster}
	\fix(\varphi^*) = \{u\in \R^n \mid u_{\varphi(1)}=u_1, \ldots, u_{\varphi(n)}=u_n\}
\end{equation}
is dynamically invariant.

If two vertices~$j,k$ have the same neighbors---such nodes are called~\emph{twins}---then the transposition interchanging~$j$ and~$k$ is a symmetry. 
Consequently, the cluster subspace~$\{u_j=u_k\}$ is dynamically invariant.


\section{Graphon dynamical systems and their symmetries} \label{sec:graphon_dynamical_system}
\label{sec:graphon_dynamics}

We now introduce graphon dynamical systems as a generalization of graph dynamical systems and extend the statements in Section~\ref{sec:graph_dynamics} to this larger class.

\begin{definition} \label{def:graphon}
Let~$J = (\Omega, \mu)$ be a probability space. 
A \emph{kernel} is a symmetric measurable function~$W: \Omega\times \Omega \to \R$.
A \emph{graphon}~$(J, W)$ is a symmetric measurable function $W: \Omega\times \Omega \to [0,1]$, that is, a kernel with range~$[0,1]$.
\end{definition}

To lighten the exposition, we will often use~$\Omega$ and~$J$ interchangeably.
On the other hand, when integrating~$\Omega$ with respect to two different probability measure,
we will make the underlying space~$\Omega$ explicit.

The class of graph dynamical systems~\eqref{eq:main:graph} considered in the previous section naturally generalize to graphons.
We consider a system of interacting units labelled by~$J$, where each unit~$x$ is associated to a state~$u_x \in \R$.

\begin{definition} \label{def:graphon_dynamical_system}
Let~$(J, W)$ be a graphon.
Let~$f,g:\R^2\to\R$ be two Lipschitz continuous functions.
We call~\emph{graphon dynamical system} the dynamical system in the ambient space~$L^1(J)$
induced by the evolution equation
\begin{equation}
\label{eq:main:graphon}
	\dot u_x = f\p{u_x, \int_{J} W(x,y) g(u_x, u_y) \d \mu (y)}, \qquad x\in J.
\end{equation}
\end{definition}
Since functions in~$L^1(J)$ are identified up to sets of measure zero, equation~\eqref{eq:main:graphon} has to be understood in the following sense:
For every~$t\in\R$ the equation holds for almost every~$x\in J$,
that is, up to a nullset (which may depend on~$t$).
Lemma~\ref{lem:existence_uniqueness} below shows that~\eqref{eq:main:graphon}
actually defines a dynamical system in~$L^1(J)$.

Kuramoto dynamics on a graph, introduced in the previous section, generalizes
to a graphon as follows:
\begin{equation} \label{eq:kuramoto}
	\dot u_x = \omega + \int_{J} W(x,y) \sin(u_y - u_x + \alpha) \ \d \mu (y).
\end{equation}
where~$\omega, \alpha \in \R$ are parameters. We come back to this example in numerical simulations.

We refer to~$J = (\Omega,\mu)$ as~\emph{index space}.
In contrast to some previous approaches to dynamical systems on graphons~\cite{Chiba2016, medvedev2014small, Chiba2019}, we consider graphons defined on a general probability space~$J$ rather than the unit interval~$J=I$ only.
There are three reasons for this.
First, some graphons have a natural underlying space~$J$ different from the unit interval,
e.g., the unit square with uniform measure or a sphere with uniform measure.
Although most of these~$J$ are standard probability spaces, and thus can be transformed
into~$I$ by an invertible measure-preserving transformation,
these transformations typically destroy regularity and symmetries,
see the example of prefix attachment graphs~\cite[Figure 11.3]{lovasz2012large} or the geodesic graphon on a torus of Figure~\ref{fig:wrong_space} below.
Second, we will see that interesting dynamically invariant subspaces on~$J$ can be understood by analyzing the dynamics on a different graphon, on another space~$J'$, and studying the
edge-preserving maps connecting~$J$ and~$J'$.
Third, allowing the index space~$J$ to be a discrete probability space provides a common framework for network dynamics on both finite graphs and graph limits.

Note also that one key difference between graph dynamical systems and graphon dynamical system is that in the latter the vertex set is endowed with a (possibly non-uniform) probability measure. 
In terms of the dynamics, this means that some vertices may be more influential than others.
Moreover, the kernel~$W$, which generalizes the adjacency matrix of a graph, can assume non-integer values between $0$~and~$1$.
One can think of a graphon~$(J,W)$ and a (possibly infinite) vertex-weighted and edge-weighted graph.

\subsection{Graph limits, graphons, and equivalence} \label{sec:graph_limit_theory}
Graphons have been introduced as limit of convergent dense graph sequences; see~\cite{lovasz2012large} for details.
Roughly speaking, a graph sequence is dense if the number of edges grows quadratically
in the number of vertices.
Convergence can be defined in terms of homomorphism densities (as we will not use this concept here, we refer to~\cite{lovasz2012large} for a definition): A dense graph sequence~$(G_n)_n$ is convergent if for every finite graph~$F$ the sequence of homomorphism densities~$(t(F, G_n))_n$ is convergent. 
One of the main results of the theory is that for every convergent dense graph sequence there is a graphon~$(J,W)$ such that~$t(F, G_n)$ converges to~$t(F, W)$.
It turns out that several~$(J,W)$ are limit of the same graph sequence.
Therefore, a suitable notion of isomorphism is necessary.

\begin{definition}
Two graphons~$(J_1,W_1)$ and~$(J_2,W_2)$ are \emph{isomorphic up to nullsets} if there is a measure preserving transformation~$\varphi: J_1\to J_2$ which is invertible up to
nullsets and satisfies~$W_2=W_1^\varphi$ almost everywhere,
\davide{where we define~$W^\varphi(x,y) = W(\varphi(x), \varphi(y))$.}
\end{definition}

This notion of isomorphism is the most relevant for our purposes since, as we will see, it preserves dynamics.

\begin{definition}
Two graphons~$(J_1,W_1)$ and~$(J_2,W_2)$ are~\emph{weakly isomorphic}
if they have the same homomorphism densities or,
equivalently, if they are limit of the same graph sequence (again see~\cite{lovasz2012large}).
\end{definition}

Two graphons that are limit of the same graph sequence can lie on completely different probability space~$J$. A consequence is that weakly isomorphic graphons may lead to wildly different graphon dynamical systems as we discuss further below.

\begin{example} \label{ex:random_graph}
For example, with probability~$1$ a sequence of \ER random graphs on~$n$ nodes converges to both the constant graphon~$W=1/2$ on~$J=I$ and the constant graphon~$W=1/2$ on the trivial probability space~$J=\{1\}$ with~$\mu(\{1\})=1$; these graphons are weakly isomorphic but not
isomorphic up to nullsets. As we will see in Section~\ref{sec:constant_graphon}
the first one supports a rich infinite-dimensional dynamical system,
while the second one is one-dimensional.
\end{example}

One of the contribution of this paper is clarifying the role of weak isomorphism in dynamics.

Defining automorphisms for graphons involves some nuances.
For $\varphi:J\to J$ write $W^\varphi(x,y) = W(\varphi(x), \varphi(y))$ as in~\cite{lovasz2012large}.
Requiring $W^\varphi(x,y) = W(x,y)$ for almost every~$x,y$ is not sufficient, as noted in~\cite{lovasz2015automorphism}:
The reason is that, being two measurable functions essentially the same up to nullsets, this definition would allow every permutation of a finite set of points to be an automorphism. 
For example, every graphon would have transitive automorphism group.
Following~\cite{lovasz2015automorphism} we define:

\begin{definition} \label{def:graphon:automorphism}
Let~$(J, W)$ be a graphon. 
A~\emph{graphon automorphism} is an invertible function~$\varphi:J\to J$ satisfying
\begin{itemize}
\item [\bf (A1)] the function $\varphi$ is measure preserving;
\item [\bf (A2)] for every~$x\in J$ and for almost every~$y\in J$
	we have~$W(\varphi(x),\varphi(y))=W(x,y)$.
\end{itemize}
We write~$\Aut(J,W)$ for the group of graphon automorphisms.
\end{definition}

Weakly isomorphic graphons can have wildly different automorphism groups.
To address this fact,
in~\cite{lovasz2015automorphism} the authors define the automorphism group on a specific class of graphons, called twin-free graphons.
Every graphon is weakly isomorphic to
a twin-free graphon~\cite[Proposition 13.3]{lovasz2012large} and the automorphism groups
of twin-free graphons are somehow better behaved.
Since weak isomorphisms do not preserve dynamics, we cannot restrict our analysis
of automorphism groups to twin-free graphons,
although they will play a central role in Section~\ref{sec:homogeneous_twins}.


\subsection{Graphon dynamical systems} \label{sec:graphon_dynamics}

We now turn back to graphon dynamical systems~\eqref{eq:main:graphon} and their properties. First note that the dynamics is well-defined:

\begin{lemma} \label{lem:existence_uniqueness}
Equation~\eqref{eq:main:graphon} defines a
dynamical system on~$L^1(J)$.
\end{lemma}
\begin{proof}
Let~$K_f, K_g \geq 0$ be Lipschitz constants for the functions~$f,g$ with respect
to the norm~$\norm{(x,y)}_1 = \abs{x}+\abs{y}$.
Let~$\mathcal F (u)$ denote the right hand side of~\eqref{eq:main:graphon}.
For every~$u,v\in L^1(J)$ we have
\begin{align*}
	\norm{\mathcal F (u) - \mathcal F (v)}_1
	& \leq
		\int_{J} K_f \p {
			\abs{u_x - v_x}
			+ K_g \int_{J} \abs{u_x - v_x}
			+ \abs{u_y - v_y} \d \mu(y)
		}
		\d \mu(x) \\
	& \leq
		\p{K_f + 2 K_f K_g} \norm{u-v}_1.
\end{align*}
Therefore the operator~$\mathcal F$ is Lipschitz continuous.
By~\cite[Theorem 7.3]{brezis2011functional} solutions exist and are unique for~$t\in[0,\infty)$.
Let~$\Phi_t: L^1(J) \times [0,\infty)\to L^1(J)$ denote the induced semiflow.

Now consider a variation of~\eqref{eq:main:graphon} in which~$f$ has been replaced
by~$-f$. The same argument applies, giving another
semiflow~$\Psi_t: L^1(J) \times [0,\infty)\to L^1(J)$.
Now fix~$t\geq 0$.
By differentiating, it is easy to see that for every~$s\in [0,t]$
the identity~$\Phi_{t-s} (u) = \Psi_{s} (\Phi_t(u))$ holds.
By taking~$s=t$ we have~$\Psi_{t} (\Phi_t(u)) = u$
and similarly one can prove that~$\Phi_{t} (\Psi_t(u)) = u$.
Therefore the semiflow~$\Phi$
can be extended to a flow by defining~$\Phi_{t} = \Psi_{-t}$ for~$t<0$.
\end{proof}

\subsubsection{Examples of graphon dynamical systems}
Recall that for a graphon~$(J,W)$, we denote by~$\mathcal D(J,W)$ the graphon dynamical system and by~$\flow^W_t(u)$ the associated {flow}. 
Note that the dynamical system depends on the probability measure~$\mu$ on the index space~$J$ as well as~$W$.
We give some examples.

\begin{example} [Dynamics on a graph] \label{ex:graph}
Let~$J$ be the set~$\{1,\ldots,n\}$ endowed with uniform probability.
Let~$\{W(x,y)=A_{x,y}\}_{x,y=1,\ldots,n}$ be the adjacency matrix of a graph.
Then~\eqref{eq:main:graphon} reduces to the dynamical system
on a graph given in equation~\eqref{eq:main:graph}:
\begin{equation}\label{eq:GraphGraphon}
	\dot u_x = f\p{u_x, \frac{1}{n} \sum_{y=1}^n A_{x,y} g(u_x, u_y)}.
\end{equation}
\end{example}

\begin{example} [Canonical Embedding] \label{ex:canonical_embedding}
Let~$A$ be the adjacency matrix of a graph with~$n$ vertices.
The \emph{canonical embedding} of the graph is a graphon on $J=I$.
Divide~$I$ into~$n$ intervals, so that~$I^2$ is divided into~$n^2$ squares.
Define $W(x,y) = A_{k,j}$ for $x\in[(k-1)/n,k/n]$, $y\in[(j-1)/n,j/n]$.
The function~$W$ associates~$0,1$ to the squares according to the adjacency matrix of the graph,
see Figure~\ref{fig:canonical_embedding}.
The resulting graphon dynamical system is distinct from~\eqref{eq:GraphGraphon}. 
However, we will see that if the initial condition is constant on the intervals $[(k-1)/n,k/n]$, $k\in\sset{1,\dotsc,n}$ then the resulting dynamics reduce to~\eqref{eq:GraphGraphon}.
\end{example}

\begin{figure}[ht]
\centering
\includegraphics{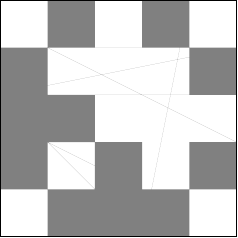}
\caption{This graphon has a block structure which is a canonical embedding
on the unit interval
of a finite graph with~$5$ vertices.}
\label{fig:canonical_embedding}
\end{figure}

\begin{example} [Dynamics on a weighted graph] \label{ex:weighted_graph}
We now generalize the previous example to weighted graphs given by a edge-weight function $W: \{1,\ldots,n\}^2\to [0,1]$ on weighted vertices.
The vertex weights are given by discrete probability distribution~$\mu=(\mu_x)_{x=1,\ldots,n}$.
The graphon dynamical system is
\[
	\dot u_x = f\p{u_x, \sum_{y=1}^n  \mu_y W_{x,y} g(u_x, u_y)}.
\]
As a concrete example, the finite graphon
\begin{center}
\begin{tikzcd}
	3/10 \arrow[dash]{r}{1} & 1/2 \arrow[dash]{r}{9/10} & 1/5
\end{tikzcd}
\end{center}
is associated to the system
\[
\begin{cases}
\dot u_1 = f(u_1, 1/2 \cdot g(u_1,u_2)) \\
\dot u_2 = f(u_2, 3/10 \cdot g(u_2,u_1) + 9/50 \cdot g(u_2, u_3)) \\
\dot u_3 = f(u_3, 9/20 \cdot g(u_3,u_2)).
\end{cases}
\]
One can define a canonical embedding of a weighted graph with vertex weights
by partitioning~$I$ into intervals of possibly different length.
\end{example}

\begin{example} [Countable graph]
In this example we will see that graphons, although typically understood as limit
objects for dense graphs, can represent dynamical systems on sparse infinite graphs,
as long as some finiteness property is satisfied.
Let~$J$ be a probability space on the set of natural numbers~$\N$.
Let~$\mu = (\mu_k)_{k\in \N}$ denote the probability measure.
For~\emph{any} graph~$G$ on~$\N$ the infinite coupled system of differential equations
\[
	\dot u_j = f\p{u_j, \sum_{k \in \N} \mu_j G_{j,k} g(u_j,u_k)}
\]
defines a dynamical system on~$L^1(J)$. An example is given
in Figure~\ref{fig:infinite_tree}.
Notice that the sum~$\sum_j \mu_j$ is finite (by definition, it is equal to~$1$);
if the vertex weights would satisfy~$\sum_{j\in\N} \mu_j = \infty$, the above
system would not be representable as a graphon system.
\end{example}

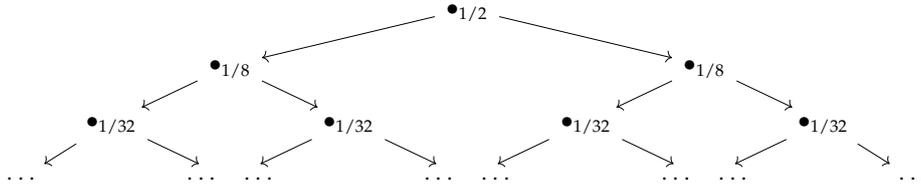
\begin{figure} [h!]
\adjustbox{scale=0.8,center}{
\begin{tikzcd} [sep=small]
	&&&& \bullet_{1/2} \arrow[dll] \arrow[drr] &&&& \\
	&& \bullet_{1/8} \arrow[dl] \arrow[dr] && &&\bullet_{1/8} \arrow[dl] \arrow[dr] && \\
	& \bullet_{1/32} \arrow[dl] \arrow[dr] &&\bullet_{1/32} \arrow[dl] \arrow[dr]&
		&\bullet_{1/32}\arrow[dl] \arrow[dr] &&\bullet_{1/32} \arrow[dl] \arrow[dr]& \\
	\cdots & \text{} & \cdots\quad \cdots & \text{}
		& \cdots\quad \cdots &\text{}
		& \cdots\quad \cdots &\text{}&\cdots \\
\end{tikzcd}
}
\caption{Infinite binary tree.}
\label{fig:infinite_tree}
\end{figure}


\subsection{Symmetries of graphon dynamical systems} 
\label{sec:graphon_dynamics_symmetries}

Theorem~\ref{thm:isomorphism:graph} shows that
an isomorphism~$G\to H$ between graphs corresponds to an isometry~$\D(H)\to \D(G)$ of (any) associated graphon dynamical system.
The main result of this section is to generalize the theorem
to graphons.
There is a key difference between graph and graphon dynamical systems: While isomorphism between graphs are necessarily invertible, for graphons they only need to be measure-preserving (but not necessarily invertible).

For completeness, we recall some basic notions. Let~$J_1=(\Omega_1, \mathcal A_1, \mu_1)$ and~$J_2 = (\Omega_2, \mathcal A_2, \mu_2)$
be two probability spaces. A \emph{measurable map} is a map~$\varphi: \Omega_1\to\Omega_2$
such that the preimage of any measurable set is measurable:
$\varphi^{-1}(A) \in \mathcal A_1$ for every~$A\in \mathcal A_2$.
Every measurable map~$\varphi:J_1\to J_2$ induces on~$J_2$ a probability, called the
\emph{push forward} of~$\mu_1$, defined by
\[
	(\varphi\#\mu_1)(A) = \mu_1(\varphi^{-1}(A)), \qquad A\in \mathcal A_2.
\]
The fundamental property of the push forward is the
change of variable formula~\cite[Lemma 1.2]{walters2000introduction}:
for every~$f\in L^1(J_1)$
\[
	\int_{\Omega_1} f(\varphi(x)) \d \mu_1 (x) = \int_{\Omega_2} f(x) \d (\varphi\#\mu_1) (x).
\]
A measurable map~$\varphi:J_1\to J_2$ is called \emph{measure preserving}
if~$\varphi\#\mu_1=\mu_2$.
Equivalently, a measurable map~$\varphi$ is measure-preserving if and only if
for every~$f\in L^1(J_1)$
\[
	\int_{\Omega_1} f(\varphi(x)) \d \mu_1 (x) = \int_{\Omega_2} f(x) \d \mu_2 (x).
\]

Notice that measure-preserving maps are not invertible in general.
We are ready for the main result of the section:

\begin{theorem} [Correspondence Theorem] \label{thm:mp_gives_isometry}
Let~$(J_1,W_1)$ and~$(J_2,W_2)$ be two graphons.
Suppose that
\[
	\varphi: J_1\to J_2
\]
is measure preserving
and~$W_1=W_2^\varphi$ holds almost everywhere.
Then
\[
	\varphi^*: L^1(J_2) \to L^1(J_1), \qquad (\varphi^* v)_x = v_{\varphi(x)}
\]
is an isometry mapping solutions to solutions.
In particular, the following diagram is commutative:
\begin{center}
\begin{tikzcd}
(J_1, W_1) \arrow[r, "\varphi"] \arrow[d] & (J_2, W_2) \arrow[d] \\
\mathcal D (J_1, W_1) & \mathcal D (J_2, W_2) \arrow[l, "\varphi^*"]
\end{tikzcd}.
\end{center}
\end{theorem}
\begin{proof}
{
Let~$\mathrm d_{L^1(J)}$ denote the distance in~$L^1(J)$.
Let~$u,v\in L^1(J_2)$.
Since~$\varphi \# \mu_1 = \mu_2$, with the change of variable formula we obtain
\[
	\mathrm d_{L^1(J_1)} (\varphi^* u, \varphi^* v)
	= \int_{J_1} \abs{\varphi^* u - \varphi^* v} \d \mu_1
	= \int_{J_2} \abs{u - v} \d \mu_2
	= \mathrm d_{L^1(J_2)} (u, v),
\]
and thus~$\varphi^*$ is an isometry.
It remains to prove that~$\varphi^*$ maps solutions to solutions.
Let~$v$ be a solution of~$\D(J_2, W_2)$. Then for almost every~$x$
\begin{align*}
	\dot v_{\varphi(x)}
	& = f\p{v_{\varphi(x)}, \int_{J_2} g(W_2(\varphi(x),y), v_{\varphi(x)}, v_y) \d \mu_2 (y)}\\
	& = f\p{v_{\varphi(x)}, \int_{J_1} g(W_2(\varphi(x),\varphi(y)), v_{\varphi(x)}, v_{\varphi(y)})
		\d \mu_1 (y)} \\
	& = f\p{v_{\varphi(x)}, \int_{J_1} g(W_1(x,y), v_{\varphi(x)}, v_{\varphi(y)})
		\d \mu_1 (y)}.
\end{align*}
Now define~$w = \varphi^* v$, that is~$w_x = v_{\varphi(x)}$ for every~$x\in J_2$.
Then~$w \in L^1(J_1)$ and by the calculation above for almost every~$x\in J_1$ we have
\[
	\dot w_{x} = f\p{u_x, \int_{J_1} g(W_1(x,y), w_x, w_y)
		\d \mu_1 (y)},
\]
and therefore~$w= \varphi^* v$ is a solution of~$\D(J_1, W_1)$.}
\end{proof}

Since~$\varphi^*$ is an isometry, it is injective and the image is closed.
The isometry~$\varphi^*$ associated to a measure preserving transformation~$\varphi$
is known as Koopman operator.

Notice that a map~$\varphi$ preserving adjacency~$W_1=W_2^\varphi$ does not
map solutions to solutions in general. Preserving measure is a key hypothesis.

As a corollary we obtain that graphons isomorphic up to nullsets support the
same dynamical systems, generalizing the analogous result for graphs:

\begin{corollary} \label{cor:imp_gives_isometry}
If the graphons~$(J_1,W_1)$ and~$(J_2,W_2)$ are isomorphic up to nullsets,
then the dynamical systems~$\D(J_1, W_1)$ and~$\D(J_2, W_2)$ are isometric.
\end{corollary}
\begin{proof}
Let~$\psi$ the inverse of~$\varphi$.
Since~$\psi$ is invertible and measure preserving, its inverse~$\psi$ is measure preserving.
We conclude by noticing that~$\psi^*$ is the inverse of~$\varphi^*$.
\end{proof}

Corollary~\ref{cor:imp_gives_isometry} shows that the index space~$J$ can be replaced by
any probability space isomorphic to~$J$ up to nullsets, without effectively changing the dynamics.
This replacement however can change the regularity of~$W$ as a function.

Recall that a graphon automorphism is \davide{a} measure preserving map~$\varphi: J\to J$
such that for every~$x$ we have~$W^\varphi(x,y) = W(x,y)$ for almost every~$y$.
In particular Corollary~\ref{cor:imp_gives_isometry} applies to graphon
automorphisms, motivating the following definition:

\begin{definition} \davide{\label{def:graphon_symmetry}}
Let~$(J,W)$ be a graphon. Let~$\varphi: J\to J$ an automorphism of the graphon.
The map~$\varphi^*: L^1(J)\to L^1(J)$ is an isometry and a symmetry
of the graphon dynamical system~$\D(J,W)$.
We call~$\varphi^*$ a \emph{graphon-induced symmetry}.
\end{definition}

By definition the graphon automorphisms of~$(J,W)$ are in one-to-one correspondence with the graphon-induced symmetries of~$\D(J,W)$.
We prefer to keep the notions separate to highlight the different action space, $J$~and~$L^1(J)$ respectively.
In later sections we will compute the automorphism group~$\Aut(J,W)$ of several graphons.

\davide{
\begin{definition} \label{def:graphon_generalized_symmetry}
Let~$(J_1,W_1)$ and~$(J_2,W_2)$ be two graphons.
Suppose that~$\varphi: J_1\to J_2$
is measure preserving and~$W_1=W_2^\varphi$ holds almost everywhere.
Then we call~$\varphi^*$ a \emph{generalized graphon-induced symmetry}.
\end{definition}
}

\subsection{Dynamically invariant subspaces and \davide{generalized} symmetries}
\label{sec:InvSubsp}
\davide{
We now discuss two ways in which dynamically invariant subspaces can arise:
as the fixed point set of a graphon symmetry,
or as the image of a generalized symmetry.}

Let~$\varphi: J\to J$ a graphon automorphism. By Corollary~\ref{cor:imp_gives_isometry}
the dual map~$\varphi^*: L^1(J)\to L^1(J)$ is a symmetry of the graphon dynamical system.
The set
\begin{equation} \label{eq:invariant_fix}
	\fix(\varphi^*) = \{u\in L^1(J) \mid u(\varphi(x)) = u(x) \ a.e.\}
\end{equation}
is a closed, dynamically invariant subspace of~$L^1(J)$.
\davide{This is a generalization of dynamically invariant cluster subspaces
on graphs~\eqref{eq:graph_cluster}.}

The second way is distinct and relates to the Koopman operator. 
Let~$\varphi$ be a map~$(J,W)\to(J',W')$ preserving both measure and adjacency, as in Corollary~\ref{thm:mp_gives_isometry}. Then
the dual map~$\varphi^*: L^1(J')\to L^1(J)$ is an isometry
embedding the dynamical system~$\D(J',W')$ into the dynamical system~$\D(J,W)$.
The set
\begin{equation} \label{eq:invariant_image}
	\varphi^*(L^1(J')) = \{ v\circ \varphi \mid v\in L^1(J') \}
\end{equation}
is a closed, dynamically invariant subspace of~$L^1(J)$.

As a particular case, take~$J'=J$ and let~$\varphi: J\to J$ be a
non-invertible measure preserving
transformation which preserves adjacency.
Although not invertible the map~$\varphi^*$ acts on~$L^1(J)$
as a symmetry of the dynamical system. \davide{}
A well known example of non-invertible measure-preserving transformation on the unit interval
is the doubling map~$x\mapsto 2x\, (\mod 1)$.
Non-invertible symmetries do not appear on graphs. Indeed
if~$J=\{1,\ldots,n\}$ with uniform probability then the measure preserving transformations
are exactly the permutations.

In the following sections we will analyze subsystems of the form~\eqref{eq:invariant_fix}
and~\eqref{eq:invariant_image}. In some cases, these two approaches will lead to two
alternative proofs of the same statement.


\section{Homogeneous coupling: twins yield clusters} \label{sec:homogeneous_twins}

In this section we consider graphons that describe homogeneous coupling
in the sense that there are `large' sets of vertices with the same neighbors.
Examples are cluster dynamics, all to all coupling and finitely many coupled populations.

\begin{definition}
Let~$(J,W)$ be a graphon. Two vertices~$x,x'\in J$ are~\emph{twins}
if the functions~$W(x,\cdot) = W(x',\cdot)$ are equal almost everywhere.
More generally, let~$A\subseteq J$ be a measurable set. We say that the elements of~$A$
are~\emph{twins} if
the map~$y\mapsto W(x,y)$ from~$J$ to~$\R$ is independent of~$x\in A$.
A graphon is~\emph{twin-free} if it contains no positive measure set of twin vertices.
\end{definition}

\begin{definition}
Let~$(J,W)$ be a graphon.
Two vertices~$x,x'\in J$ are \emph{twins} if~$W(x,y)=W(x',y)$ for almost every~$y\in J$.
More generally, the elements of a measurable set~$A \subseteq J$ are \emph{twins}
if for every~$x,x'\in A$ we have~$W(x,y)=W(x',y)$ for almost every~$y\in J$.
A graphon is~\emph{twin-free} if it contains no positive measure set of twin vertices.
\end{definition}

Put differently, for a set of twins~$A$ the map~$y\mapsto W(x,y)$, as a map~$J\to \R$,
is independent on~$x\in A$.

\begin{figure} [ht]
\centering
\includegraphics[width=0.35\textwidth]{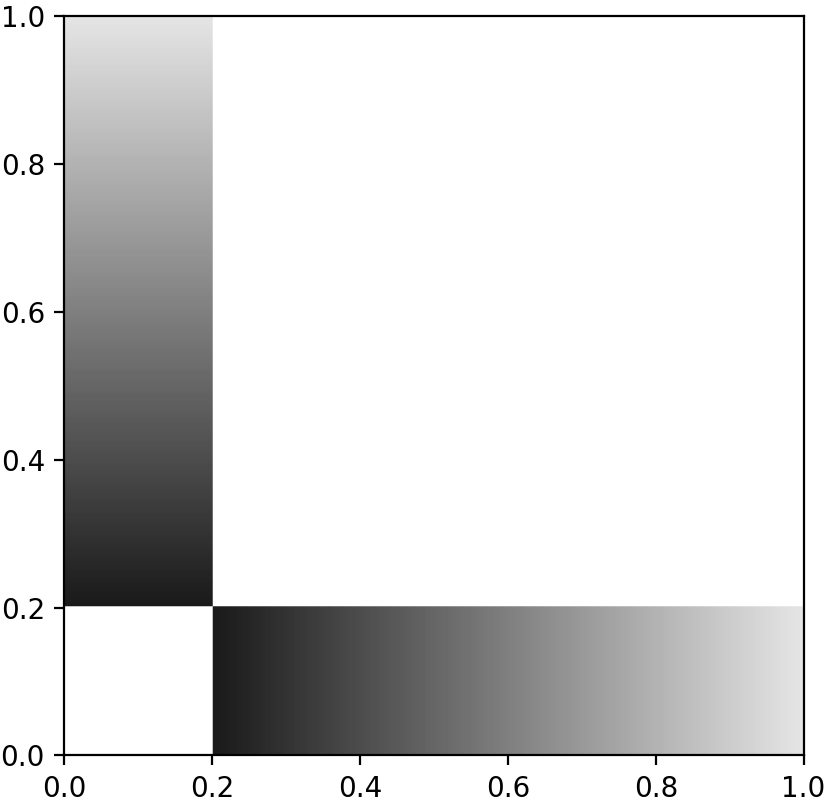}
\caption{The connection strength between a vertex~$x\in [0,1/5]$
and a vertex~$y \in (1/5,1]$ depends on~$y$ but not on~$x$.
}
\label{fig:twins}
\end{figure}

The twin relation is an equivalence relation. Moreover, the sets of the
partition are measurable. The index space~$J$ is union
of measurable sets of twins.
For example, in the graphon of Figure~\ref{fig:twins} the interval~$[0,1/5]$ is a set
of twins; in this case the partition is formed by~$[0,1/5]$ and
the singletons~$\{x\}$ for~$1/5<x\leq 1$.

\begin{definition}
Let~$A\subseteq J$ be a set of positive measure. The~\emph{cluster subspace}~$\cluster(A)$
associated to~$A$ is the subset of functions~$u\in L^1(J)$
such that~$u$ is constant on~$A$ up to nullset.
\end{definition}

\davide{
The following results, Theorem~\ref{thm:twins_constant},
shows that twin vertices starting synchronized remain synchronized over time.
In the case of finite graphs one can prove this fact by comparing the evolution
equations of a pair of twin vertices.
The context of graphons poses some technical challenges.
First, we cannot look at only two vertices at once:
since functions in~$L^1$ are identified
up to nullsets, for every fixed~$x,y\in L^1([0,1])$ we have
\[
	\{u \in L^1([0,1]) \mid u_x=u_y\} = L^1([0,1]).
\]
It is necessary to work with a set of twins~$A$ of positive measure at once,
rather than individual pairs.
Second, the evolution equation~\eqref{eq:main:graphon} holds for every~$t$
and almost every~$x\in J$, that is, for every~$t$ there is a nullset~$K_t$ such that
the equation holds for every~$x \in J\setminus K_{t}$.
However, we cannot fix a single set of measure zero that works for all~$t$:
The set~$K_t$ depends on~$t$, and the set~$\bigcup_t K_t$ can have positive measure.
}

We give two proofs of Theorem~\ref{thm:twins_constant},
one in the general case and one in the special case of the unit
interval~$J=I$.
The proofs are based on the two different applications~\eqref{eq:invariant_fix}
and~\eqref{eq:invariant_image} of the Correspondence Theorem~\ref{thm:mp_gives_isometry},
as announced in Section~\ref{sec:InvSubsp}.

\begin{theorem} \label{thm:twins_constant}
Let~$(J,W)$ be a graphon and~$A\subseteq J$ a set of positive measure.
Suppose that all the vertices in~$A$ are twins.
Then the cluster subspace~$\cluster(A)$ is (forward and backward) dynamically invariant.
\end{theorem}

\begin{proof}
This first proof is based on~\cite[Proposition 13.3]{lovasz2012large}.
Let~$J = (\Omega, {\mathcal A},\mu)$. Let~$\widetilde {\mathcal A}$
be the sigma-algebra of those sets in~$\mathcal A$ that do not separate any point of~$A$.
Let~$\widetilde W = E(W \mid \widetilde {\mathcal A} \times \widetilde {\mathcal A})$
the conditional expectation of the function~$W$
with respect to~$\widetilde {\mathcal A} \times \widetilde {\mathcal A}$.
Let~$\widetilde J = (\widetilde \Omega, \widetilde {\mathcal A}, \tilde \mu)$
be the quotient probability space obtained by identifying the elements of~$A$ and let~$\varphi: J\to \widetilde J$
denote the projection.

Then~$\varphi$ is measure preserving and satisfies~$\widetilde {W}^\varphi = W$ almost everywhere, see~\cite[Proposition 13.3]{lovasz2012large} for details.
By the Correspondence Theorem~\ref{thm:mp_gives_isometry} the map~$\varphi^*$
embeds~$\mathcal D(\widetilde J, \widetilde W)$ isometrically into~$\mathcal D(J,W)$.
The image~$\varphi(L^1(\widetilde J))$ is the subset of~$L^1(J)$ consisting of functions
of the form~$v\circ {\varphi}$ for some~$v\in L^1(\widetilde J)$. These
are exactly the functions constant on~$A$ up to nullset.
\end{proof}

In the special case of the unit interval~$J = I$ we can give a somewhat simpler proof,
based on ergodic transformations.

\begin{proof} [Proof of Theorem~\ref{thm:twins_constant} in the case $J=I$]
The unit interval is an atomless standard probability space.
Every measurable subset of a standard probability space is
standard~\cite{rokhlin1949fundamental,haezendonck1973abstract}.
Therefore~$A$ endowed with the normalized measure is an atomless standard probability space.
All atomless standard probability spaces are isomorphic up to nullset
to the unit interval~\cite{halmos1942operator}.

In particular~$A$ supports an ergodic transformation~$\gamma: A\to A$.
Define~$\gamma$ on~$I\setminus A$ as the identity map.
A function in~$L^1(I)$ is $\gamma$-invariant if and only if its restriction on~$A$
is $\gamma$-invariant and, since $\gamma$ is ergodic on~$A$,
if and only if it constant over~$A$~\cite[Theorem 1.6]{walters2000introduction}.
\end{proof}

Let us return to Figure~\ref{fig:twins}.
By Theorem~\ref{thm:twins_constant} the set of functions constant over~$[0,1/5]$
up to nullsets is dynamically invariant, that is, it is a cluster.
On the other hand, notice that
the set of functions constant over~$[1/5,1]$ up to nullsets is not dynamically invariant.

In Theorem~\ref{thm:twins_constant} we reduce the space by identifying a set of twins.
However, inspired by~\cite{lovasz2012large},
one can identify all the set of twins in~$(J,W)$,
obtaining the~\emph{twin-free quotient}~$(\widetilde J, \widetilde W)$.
The sigma-algebra and probability on the quotient are constructed as in the first proof of~Theorem~\ref{thm:twins_constant}.
Dynamics on the twin-free quotient is given by:

\begin{proposition} \label{prop:twin_free_core}
Let~$(J,W)$ be a graphon and~$(\widetilde J, \widetilde W)$ the associated
twin-free quotient. Then
\begin{equation} \label{eq:twin-free_core}
	\ppp{
			u\in L^1(\davide{J}) \mid u_x = u_y \text{ for all twins } x,y
		}
\end{equation}
is a closed, dynamically invariant subspace of~$J$.
The dynamics of~$\D(J,W)$ restricted to~\eqref{eq:twin-free_core}
is isometric to the dynamics on the twin-free quotient~$\D(\widetilde J, \widetilde W)$.
\end{proposition}
\begin{proof}
Apply the Correspondence Theorem~\ref{thm:mp_gives_isometry} to the
projection~$J\to \widetilde J$.
\end{proof}

This motivates the following definition:
\begin{definition}
We call~\eqref{eq:twin-free_core} the~\emph{twin-sync} subspace of~$\D(J,W)$.
\end{definition}

The twin-sync subspace is the subset in which any two twins share the same state.
Proposition~\ref{prop:twin_free_core} shows that dynamics on the twin-sync subspace
is the same as the dynamics on the twin-free quotient.

Recall that graphons isomorphic up to nullsets lead to isometric graphon dynamical systems
(Corollary~\ref{cor:imp_gives_isometry}) while, on the other hand, weakly isomorphic
graphons may have wildly different graphon dynamical systems (Example~\ref{ex:random_graph}).
The following result explains the exact role of weak isomorphism in dynamics.

It requires the technical hypothesis that the probability space~$J$ is standard.
This \davide{} however is not very restrictive,
as most spaces considered in practice are standard:
finite and countable discrete probability spaces,
the unit interval with Lebesgue measure,
any absolute continuous distribution on~$\R^n$,
spheres and tori with uniform probability,
the set of all continuous functions~$[0,\infty) \to \R$ with the Wiener measure.

\begin{corollary} \label{cor:weak_isom_twin_sync}
Suppose that~$J_1$ and~$J_2$ are standard probability spaces.
If the graphons~$(J_1,W_1)$ and~$(J_2,W_2)$ are weakly isomorphic,
then the twin-sync subspaces of~$\D(J_1,W_1)$~and~$\D(J_2,W_2)$
have isometric dynamics.
\end{corollary}
\begin{proof}
Since the graphons~$(J_1,W_1)$ and~$(J_2,W_2)$ are weakly isomorphic
and every graphon is weakly isomorphic to its twin-free realization,
then the graphons~$(\widetilde J_1, \widetilde W_1)$ and~$(\widetilde J_2, \widetilde W_2)$
are also weakly isomorphic. Twin-free weakly isomorphic graphons on standard probability spaces
are isomorphic up to nullset~\cite[Theorem 13.9]{lovasz2012large}.
By Corollary~\ref{cor:imp_gives_isometry} graphons isomorphic up to nullset have
isometric dynamics.
\end{proof}

\begin{remark} \label{rem:comparing_sameness}
Roughly speaking, replacing a vertex by a set of twins and adjusting the measure accordingly
leads to an equivalent graphon in the sense of graph limit theory (indeed,
weakly isomorphic),
which supports a different dynamic unless twins start synchronized.
Corollary~\ref{cor:weak_isom_twin_sync} shows that this is essentially the only difference
between the notion of graphon equivalence in combinatorics
and in dynamics.
\end{remark}

\subsection{Identical all-to-all coupling: The constant graphon} \label{sec:constant_graphon}

Consider a constant graphon. In this case all vertices are twins.
This has an immediate consequence for dynamics:
By Theorem~\ref{thm:twins_constant}, for every positive-measure subset~$A$ the cluster
subspace~$\cluster(A)$ is dynamically invariant.
In this section we analyze the symmetry group of the graphon dynamical system and understand further consequences for the dynamics.

For the sake of example, we consider~$J=I$ and~$W=1$.
The automorphism group of the graphon~$(I,1)$ is the full group of measure preserving transformations~$\mp(I)$ of the unit interval~$I$, which has been
well studied as a topological group~\cite{fathi1978groupe, halmos2017lectures, nhu1990group}.
In our context it is interesting to remark that
every element of~$\mp(I)$ can be approximated by a permutation of a finite partition of~$I$
into intervals~\cite{kloeden1997constructing}
(although the typical element of~$\mp(I)$ is
not of this form~\cite{chaika2018typical}).

By any means, the group~$\mp(I)$ is large.
This leads to a large number of dynamically invariant subspaces (see Section~\ref{sec:InvSubsp}).
We describe some:

\begin{proposition} \label{prop:autom_constant}
Consider the graphon dynamical system~$\D(I,1)$.
The following subsets of~$L^1(I)$ are dynamically invariant:
\begin{enumerate} [(i)]
\item For every measurable subset~$A\subseteq I$,
	the cluster subspace~$\cluster(A)$;
	\label{prop:autom_constant:constant}
\item For every measurable subsets~$A_1,\ldots,A_n \subseteq I$,
the subspace \[ \bigcap_{k=1}^n \cluster(A_k) \] in which every~$A_k$ is cluster;
	\label{prop:autom_constant:partition}
\item The set of functions that are injective up to nullset;
	\label{prop:autom_constant:injective}
\item For every positive integer~$q$, the set of almost everywhere $(1/q)$-periodic functions;
	\label{prop:autom_constant:periodic}
\item The set of functions satisfying the identity~$u_x = u_{1-x}$ almost everywhere.
	\label{prop:autom_constant:even}
\end{enumerate}
\end{proposition}
\begin{proof}
Part \eqref{prop:autom_constant:constant} is a particular case of
Theorem~\ref{thm:twins_constant}.
Part~\eqref{prop:autom_constant:partition} and
part \eqref{prop:autom_constant:injective} follow from part \eqref{prop:autom_constant:constant}
and the fact that the set of dynamically invariant
subsets is closed by intersection, union and complement.

We prove \eqref{prop:autom_constant:periodic} in two ways.
The transformation~$\varphi: x\mapsto x+1/q\, (\mod 1)$ is
invertible and measure preserving, thus a symmetry
of the dynamical system by Corollary~\ref{cor:imp_gives_isometry}.
The fixed point set~$\fix(\varphi^*) \subseteq L^1(I)$
is the set of $(1/q)$-periodic functions.
Alternatively, consider the non-invertible
measure preserving transformation~$\varphi: x\mapsto qx\, (\mod 1)$
and consider the image of the dual map~$\varphi^*$.

We prove \eqref{prop:autom_constant:even} in two ways.
The transformation~$\varphi^*: x\mapsto 1-x$ is invertible and measure preserving, thus a symmetry
of the dynamical system by Corollary~\ref{cor:imp_gives_isometry}.
The fixed point set~$\fix(\varphi^*) \subseteq L^1(I)$
is the set of functions satisfying~$u_x = u_{1-x}$ for almost every~$x\in I$.
Alternatively, consider the non-invertible
measure preserving transformation~$\varphi$
mapping~$x\mapsto 2x$ if~$x\leq 1/2$ and $x\mapsto 1-2x$ if~$x>1/2$
and consider the image of the dual map~$\varphi^*$.
\end{proof}

By taking~$A=I$ in Proposition~\ref{prop:autom_constant}~\eqref{prop:autom_constant:constant}
we obtain the twin-synch subspace of the system: Since all vertices are twins, dynamics reduces
to the $1$-dimensional ordinary differential equation
\[
	\dot u = f(u, g(u,u)), \qquad u\in \R.
\]
This describes the dynamics of the system as one giant cluster.

Proposition~\ref{prop:autom_constant}~\eqref{prop:autom_constant:partition}
leads to the finite-dimensional system
\[
	\dot u_j = f \p{u_j, \sum_{y=1}^n g(u_j, u_k) \mu_k}, \qquad j=1,\ldots,n
\]
where~$1,\ldots,n$ are the labels of the subsets,~$(\mu_k)_k$ their measure and~$u_j$
the state of the cluster.
This describes multi-population cluster dynamics.
Notice that if~$\mu_1=\cdots=\mu_n$ then we obtain the graph dynamical system on a complete graph~$K_n$, and that
there are uncountably many ways of partitioning~$I$ into~$n$ sets of measure~$1/n$,
thus uncountably many copies of~$\D(K_n)$ in~$\D(I, 1)$.

Proposition~\ref{prop:autom_constant}~\eqref{prop:autom_constant:periodic} leads to
\[
	\dot u_x = f \p{u_x, q \int_0^{1/q} g(u_x, u_y) \d y}, \qquad x\in [0,1/q].
\]
By substituting~$q\d y = \d \mu (y)$, the normalized Lebesgue measure on~$[0,1/q]$,
we see that this system is isometric to the full system~$\D(I,1)$.
This shows that the dynamical system~$\D(I,1)$ contains infinitely many copies of itself.

\subsubsection{Canonical invariant region}
Proposition~\ref{prop:autom_constant} part~\eqref{prop:autom_constant:injective}
generalizes an important property known for systems
on complete graphs~\cite{ashwin2016identical, ashwin1992dynamics}.
Since in a complete graph any two vertices are twins, the states of two vertices can never
cross each other over time. In particular, if the initial condition~$k\mapsto u_k(0)$ is injective,
then~$k\mapsto u_k(t)$ is injective for every~$t$.
As a consequence, one can restrict dynamics to the \emph{canonical invariant region}
defined by~$u_1 < u_2 < \ldots < u_n$, see~\cite{ashwin2016identical, ashwin1992dynamics}.
Proposition~\ref{prop:autom_constant} part~\eqref{prop:autom_constant:injective}
shows that, in the context of graphons, injectivity is preserved up to nullset.

\subsubsection{Large random graphs}
For each positive integer~$n$ consider an \ER random graph with~$n$ vertices.
The probability that an \ER random graph has a non-trivial automorphism
goes to~$0$ as~$n$ goes to infinity~\cite[Corollary~2.3.3]{godsil2001algebraic}.
On the other hand, with probability~$1$ the graph sequence converges
to the constant graphon~$(I, 1/2)$, which has a large group of symmetries.
Therefore, although the system on~$n$ vertices is non-symmetric,
one can expect dynamics to resemble the symmetries of the limit for~$n$ large.
In Section~\ref{sec:approximate_symmetries} we will prove that this is indeed the case.


\subsection{Coupled populations: Graphons with block structure} \label{sec:block_structure}

Consider a partition of~$J$ into finitely many measurable subsets~$J_1,\ldots,J_n$.
In this section we suppose that the coupling between any two vertices~$x\in J_j$ and~$y\in J_k$
depends only on~$j$ and~$k$, that is, for every~$k=1,\ldots,n$ the vertices
in each~$J_k$ are all twins.
This assumption models a family of~$n$ populations
with homogeneous coupling within a population but not necessarily among the populations.
The case~$n=1$ is covered in Section~\ref{sec:constant_graphon}.

For the sake of example, we consider~$J=I$.
Fix~$n>1$ and define~$I_k = [k-1/n, k/n]$ for~$k=1,\ldots,n$.
Let~$G$ be a graph with~$n$ vertices and let~$G_{j,k}$ denote the adjacency matrix.
Dynamics on~$G$ is given by
\begin{equation} \label{eq:blocks:finite}
	\dot u_k = f\p{u_k, \frac{1}{n} \sum_j G_{j,k} g(u_j, u_k)}.
\end{equation}
We now define a graphon~$W_G$ on~$I$ by coupling the partition intervals according to~$G$:
\begin{equation} \label{eq:blocks}
	W_G = \sum_{j,k} G_{j,k} \1_{I_j\times I_k},
\end{equation}
see Figure~\ref{fig:canonical_embedding}.
Dynamics on~$W_G$ are given by
\begin{equation} \label{eq:blocks:infinite}
	\dot u_x = f\p{u_k, \sum_{j} G(j,k) \int_{I_j} 
		g(u_x, u_y) \d y}, \qquad \text{if } x\in I_k.
\end{equation}
The system~\eqref{eq:blocks:finite} can be obtained from~\eqref{eq:blocks:infinite}
by restricting dynamics to the subspace~$\bigcap_{k=1}^n \cluster(I_k)$
in which each~$I_k$ is cluster: 
One can show this directly from the equations or
by applying the Correspondence Theorem~\ref{thm:mp_gives_isometry}
to the function~$I \to \{1,\ldots,n\}$ mapping each point~$x\in I_k$ to the label~$k$
of the interval it belongs to. 

The graph-induced symmetries of~\eqref{eq:blocks:finite} are the automorphism of~$G$.
These extend naturally to graphon-induced symmetries as permutations of the partition intervals.
However the graphon system~\eqref{eq:blocks:infinite} has additional symmetries.
Indeed, every interval~$I_k$ supports its own group of invertible measure preserving transformations~$\mp(I_k)$.
Notice that the groups~$\mp(I_k)$ are all isomorphic to~$\mp(I)$ as groups (but the isomorphism does not preserve the measure).
Since~$\Aut(G)$ acts on~$\prod_k \mp(I_k)$ by exchanging the intervals, this formally gives the symmetry group the structure of the wreath product
\[
	\prod_k \mp(I_k) \rtimes \Aut(G) \cong \mp(I) \wr \Aut(G).
\]
We will see in Proposition~\ref{prop:block_structure_group} that,
if the adjacency matrix of~$G$ is invertible, then this is indeed the full
automorphism group of~$W_G$. In general the group can be larger.
Graphs with invertible adjacency matrix are analyzed in~\cite{sciriha2007characterization}.
If~$G$ contains twins, say~$j$ and~$k$, then the adjacency matrix is not invertible
and the automorphism is allowed to exchange mass between the intervals~$I_j$ and~$I_k$.

\begin{proposition} \label{prop:block_structure_group}
Let~$G$ be a graph with invertible adjacency matrix.
Then the automorphism group of $W_G$ is
\[
	\Aut(W_G) \cong \prod_k \mp(I_k) \rtimes \Aut(G).
\]
\end{proposition}
\begin{proof}
It remains to be shown that every graphon automorphism is a composition
of an interval permutation (corresponding to an automorphism of~$G$)
and measure preserving transformations of the intervals into themselves.
To show this, observe that
in general a measure preserving transformation of~$I$ transfer\davide{s} mass between the
intervals~$(I_k)_k$. This transfer is represented by a Markov chain.
More precisely, we will obtain a double stochastic matrix~$P$.
Under the stronger assumption that the transformation is an automorphism
of~$W_G$, we will show that all the entries of~$P$ are either~$0$ or~$1$.
That is, the transformation preserves the partition.

Let~$(G(p,k))_{p,k}$ denote the adjacency matrix associated to~$G$.
Fix~$x$. By definition we have~$W(\varphi(x),\varphi(y)) = W(x,y)$ for almost every~$y$.
Suppose that~$x\in I_p$ and~$\varphi(x)\in I_q$.
By definition we obtain:
\[
	G(p,k) \1_{I_k \cap \varphi^{-1} (I_j)}(y)
		=
	G(q,j) \1_{I_k \cap \varphi^{-1} (I_j)}(y).
\]
Define~$p_{kj} = \abs{I_k \cap \varphi^{-1} (I_j)}/(1/n)$ and the matrix~$P=(p_{kj})_{k,j}$.
Notice that~$P$ is double stochastic. Moreover, we have
\[
	G(p,\cdot) = P G(q,\cdot)
\]
and
\[
	P^\textsf{T} G(p,\cdot) = G(q,\cdot).
\]
It follows that~$P P^\textsf{T}$ and~$P^\textsf{T} P$ act identically on the
rows of the adjacency matrix of~$G$.
Since the adjacency matrix~$G$ is non-singular it follows
that~$P P^\textsf{T} = P^\textsf{T} P$ is equal to the identity matrix.
We conclude that that~$P$ is a permutation matrix.

This shows that~$\varphi$ preserves the partition. Therefore it must be the composition of
a permutation of the set of intervals and measure preserving transformations within each interval.
\end{proof}

It is interesting to look at the dynamics on the invariant subspaces.
Since each group~$\mp(I_k)$ contains an ergodic transformation,
dynamics on the fixed point set~$\fix(\prod_k \mp(I_k))$
is the same as the graph dynamics~$\D(G)$.

On the other hand, let us write~$j \sim k $ if and only if~$j,k\in \{1,\ldots,n\}$
belong to the same~$\Aut(G)$-orbit.
Then the fixed point set of~$\Aut(G)$ is the subset of~$L^1(I)$ of functions that repeat the same
values on every intervals of the same orbit:
\[
	\fix(\Aut(G)) = \ppp{u\in L^1(I) \bigm| u_{\frac{j}{n}+x} = u_{\frac{k}{n}+x},
		\, \forall x\in [0,1/n], \, \forall j\sim k}.
\]
In particular, if~$\Aut(G)$ is transitive then the fixed point set is the subset of
$(1/n)$-periodic functions and dynamics reduces to
\[
	\dot u_x  = f\p{ u_x, \deg(G) \int_0^{1/n} g(u_x, u_y) \d y}, \qquad x\in [0,1/n],
\]
where~$\deg(G)$ is the degree common to every vertex of~$G$.
Up to isomorphism this is the same as dynamics on the constant graphon~$\D(I, \deg(G)/n)$.
Intuitively, this means that if~$\Aut(G)$ is transitive and all the populations start
with the same initial conditions,
then they evolve as just one population.


\section{Networks with distance-dependent coupling} 
\label{sec:geodesic_coupling}

In this section we consider network dynamical systems that arise if the dynamical units are placed on a manifold and coupling depends on the distance of the units on the manifold. We refer to such a framework as~\emph{geodesic coupling}.
The resulting graphons are known are \emph{geodesic graphons}
or~\emph{geometric graphons}~\cite{delmas2020infinite}.
This setup will lead to specific graphon dynamical systems where the graphon~$(J,W)$ reflects the ``geometry'' of the manifold.

Graphon dynamical systems with geodesic coupling include a wide range of relevant systems.
The classic example is distance-dependent coupling on a circle: Units are indexed by $J=\T$ (equipped with the standard metric and the uniform measure) and two units $u_x$~and~$u_y$, where $x,y\in\T$, interact if and only if their distance~$\d(x,y)$ is less or equal than some fixed constant~$\delta$.
Note that if~$\delta=\pi$ the circle~$\T$ is all-to-all connected and if~$\delta<\pi$ the neighborhood of a point~$x$ is an arc of length~$2\delta$.
Such distance-dependent interactions has been considered, for example, in~\cite{Chiba2016, medvedev2014small} as the continuum limit of corresponding families of graphs ~\cite{lu2004characterizing, wiley2006size}.
Indeed, geodesic coupling on the $2$-dimensional and $3$-dimensional torus has been extensively considered in the analysis of chimera states~\cite{panaggio2013chimera, omel2019chimerapedia, maistrenko2015chimera}.
Geodesic coupling on the $2$-dimensional sphere~$S^2$ has been considered in graph limit theory~\cite{lovasz2012large}, see Example~13.2 and Example~13.16.

Graphon dynamical systems with geodesic coupling come with a natural choice for the index space~$J$---the manifold over which the coupling is defined. 
With this index space, the kernel~$W$ will have a simple form determined by the coupling, reflecting any symmetries the space~$J$ may have.
Thus, while fixing $J=I$ is sufficient in graph limit theory~\cite{lovasz2012large},
for graphon dynamical systems with geodesic coupling it can obscure geometric (symmetry) properties---especially if the topological structure has dimension larger than~$1$.

As an example, consider a geodesic coupling graphon on~$J=\T^2$.
The probability spaces~$I$ and~$\T^2$ are isomorphic up to nullset. An explicit
isomorphism is the measure preserving transformation $I\to\T^2$
that separates odd and even digits of the binary expansion.
Representing the graphon on~$J=I$, however,
hides the symmetries, see Figure~\ref{fig:wrong_space}.

\begin{figure}[h]
\centering
\includegraphics[width=0.5\textwidth]{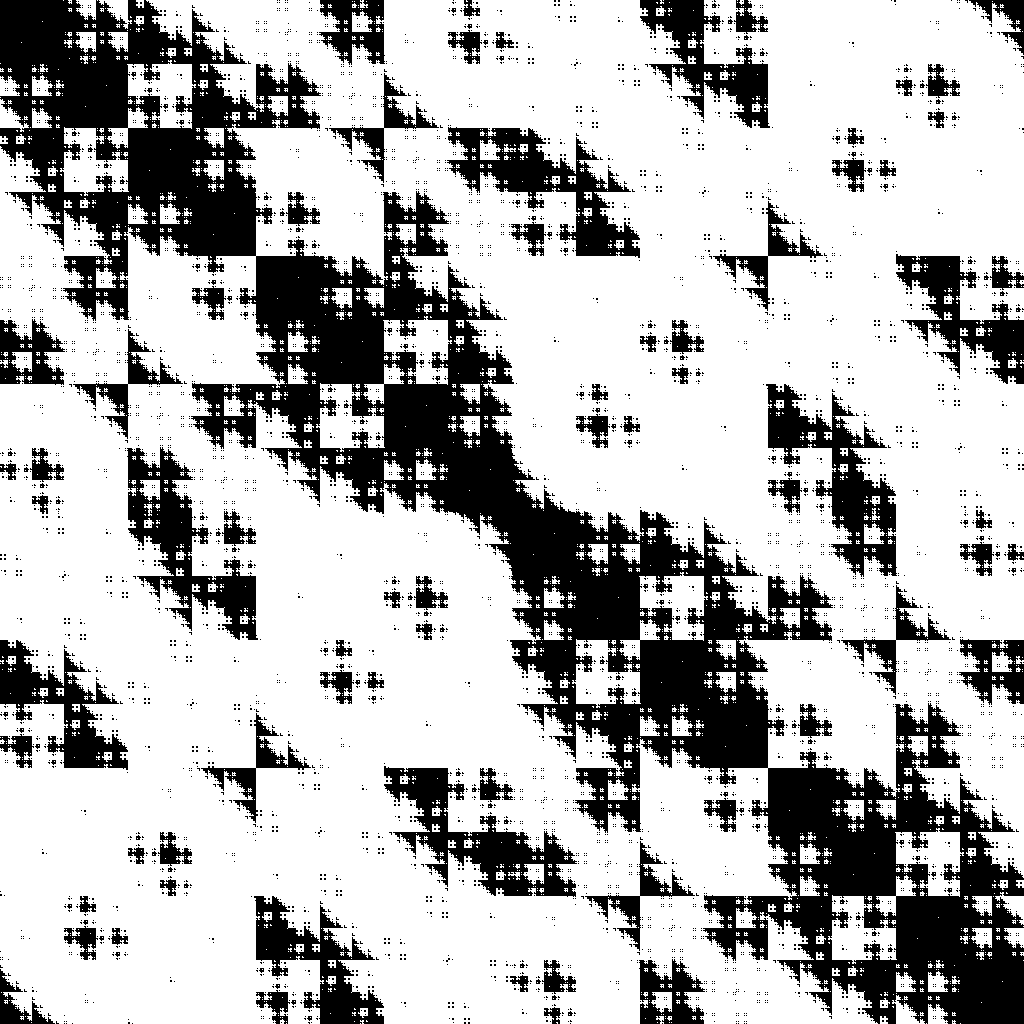}.
\caption{
Representing a naturally $2$-dimensional index space on the $1$-dimensional
unit interval can create complicated, fractal-like pictures,
which hide the underlying graphon symmetries.
The graphon in this figure has automorphism group~$\Dih_4 \ltimes \T^2$.
}
\label{fig:wrong_space}
\end{figure}

This section is organized as follows: First, we frame geodesic coupling in the context of graphons;
second, we apply the graphon formalism in order to understand symmetries;
third, we present numerical simulations for the $2$-torus and the $2$-sphere.
As a side product of our analysis we are able to explicitly compute the automorphism group
of several geodesic graphons.

\subsection{Geodesic graphons}
In order to frame geodesic coupling in the context of graphons, we require
the notion of Hausdorff measure.
Let~$J$ be a compact, connected $d$-dimensional manifold endowed with metric~$\d$
and $d$-dimensional Hausdorff probability measure~$\mu$.
If~$J$ is the unit circle~$\T$, the distance~$\d(x,y) = \measuredangle (x,y)$
is equal to the \davide{smaller} angle (in radians) between~$x$ and~$y$
while the measure~$\mu$ coincides with normalized Lebesgue measure inherited
from the interval~$[0,2\pi]$.
More generally, if~$J$ is the $d$-dimensional torus~$\T^d = \R^d/\Z^d$, metric and
measure are inherited from~$\R^d$.
If~$J$ is the unit $2$-dimensional sphere~$\S^2$, the distance~$\d$ is the arc length and~$\mu$
is the spherical measure.
The Hausdorff measure is defined from the metric.
A consequence of this fact is that any isometry of the manifold is measure preserving.

Fix~$\delta>0$. The \emph{geodesic graphon} $W_\delta: J\times J\to \{0,1\}$ is defined as
follows:
\[
	W_\delta(x,y) =
	\begin{cases}
		1, \quad \text{if $\d(x,y)\leq \delta$} \\
		0, \quad \text{otherwise}.
	\end{cases}
\]
Notice that~$W_\delta$ is a graphon:
From the definition of Hausdorff measure it follows that~$W_\delta$ is $\mu$-measurable
and clearly~$W_\delta(x,y)=W_\delta(y,x)$.
The neighborhood of the vertex~$x\in J$, in the sense of graphon adjacency,
is the ball of radius~$\delta$ centered at~$x$.


\subsection{The automorphism group of geodesic graphons}
Let~$W_\delta$ be a geodesic graphon.
An isometry is a bijection~$J\to J$ preserving the metric~$\d$.
By definition it follows that both~$\mu$ and~$W_\delta$ are preserved by isometry.
In particular any isometry of~$J$ is a graphon automorphism of~$W_\delta$:

\begin{proposition} \label{prop:iso_in_auto}
Let~$\Iso(J)$ be the isometry group of~$J$.
For every~$\delta\geq 0$ the inclusion~$\Iso(J)\subseteq \Aut(W_\delta)$ holds.
\end{proposition}

In Theorem~\ref{thm:iso_is_auto} we will prove that, for a certain class of manifolds,
isometries are the only graphon automorphisms.
This allows us to compute the automorphism group~$\Aut(W_\delta)$
for~$J=\T^d$ and~$J=\S^d$, covering the cases considered in literature.

For the proof of Theorem~\ref{thm:iso_is_auto} to work, we need the following assumption:
The volume of the intersection of two geodesic balls depends only
on the distance between the centers and the radii.
Spaces with this property have been characterized in~\cite{csikos2011volume}.
It turns out that under mild hypothesis they are the same as harmonic spaces.
See~\cite{berndt1995generalized} for equivalent definitions of harmonic space.
Spheres and tori are harmonic spaces.

\begin{theorem} \label{thm:iso_is_auto}
Suppose that~$J$ has the property that
the volume of the intersection of two geodesic balls depends only
on the distance between the centers and the radii.
Then for every~$\delta>0$ small enough
the identity~$\Iso(J)=\Aut(W_\delta)$ holds.
\end{theorem}
\begin{proof}
We need to prove that every graphon automorphism~$\varphi$ is an isometry.
Let~$B^\delta_x$ denote the geodesic ball of center~$x$ and radius~$\delta$.
\davide{For every~$x,y\in S^1$ we have}
\[
	B^\delta_x \cap B^\delta_y =
		\varphi^{-1} \p{\varphi\p{B^\delta_x} \cap \varphi\p{B^\delta_y}}.
\]
Since~$\varphi$ preserves neighbors we have
\[
	B^\delta_x \cap B^\delta_y =
		\varphi^{-1} \p{B^\delta_{\varphi(x)} \cap B^\delta_{\varphi(y)}}.
\]
Moreover, since~$\varphi$ preserves the measure we have
\[
	\mu\p{B^\delta_x \cap B^\delta_y} = \mu \p{B^\delta_{\varphi(x)} \cap B^\delta_{\varphi(y)}}.
\]

By hypothesis the volume of the intersection of two geodesic balls depends only
on the distance between the centers and the radii.
Let~$m^\delta(\lambda)$ denote the measure of the intersection of two balls of radius~$\delta$
and center distance~$\lambda$.
Notice that if~$\delta$ is larger than the diameter of the space~$J$,
then~$m^\delta(\lambda)$ is constant in~$\lambda$:
any two balls of radius~$\delta$ have intersection equal to the whole space~$J$
and~$m^\delta(\lambda)=1$ for every~$\lambda$.
On the other hand, for~$\delta$ small enough the function~$\lambda \mapsto m^\delta(\lambda)$
is decreasing and strictly decreasing if~$\lambda \in [0,2\delta]$.
We conclude that, for every~$\delta$ small enough, we have~$\d(x,y)=\d(\varphi(x), \varphi(y))$
for any two points~$x,y$ of distance~$\d(x,y)\leq 2\delta$.

It remains to extend~$\d(x,y)=\d(\varphi(x), \varphi(y))$
to points of arbitrary distance. To do so,
choose a geodesic path between $x$~and~$y$, then divide the path in segments~$[x_i,x_i+1]$
of length smaller than~$2\delta$ with~$x_0=x$ and~$x_m=y$. By triangular inequality
we have
\[
	\d(\varphi(x), \varphi(y)) \leq \sum_i \d(\varphi(x_i), \varphi(x_{i+1}))
		= \sum_i \d(x_i, x_{i+1}).
\]
Since we have chosen a geodesic path between~$x$~and~$y$, the right
hand side is equal to~$\d(x,y)$.
This proves~$\d(\varphi(x), \varphi(y))\leq \d(x,y)$ for every pair of points.
Non-expansive local isometry of a connected compact space to itself
is an homeomorphism is a surjective map, and an isometry~\cite[Theorem 4.2]{calka1982local}.
\end{proof}


\subsection{Graphon on the torus} \label{sec:toral_graphon}
Consider the $2$-dimensional torus~$\T^2 = \R^2/\Z^2$, endowed with metric and measure induced
from~$\R^2$. Fix~$\delta<1$ and consider the geodesic graphon~$W_\delta$.
By Theorem~\ref{thm:iso_is_auto} the automorphism group of~$W_\delta$ is the same as the isometry
group of the torus. With $\Dih_k$ denoting the dihedral group on~$k$ elements, it is well known that
\[
	\Iso\p{\T^2} \cong \Dih_4 \ltimes \T^2.
\]
The component~$\T^2$ is the group of translation of the torus, acting on itself.
The dihedral group~$\Dih_4$ is generated by the transformations~$(x,y)\mapsto(-y,x)$
and~$(x,y)\mapsto(x,-y)$.
Notice that~$\Dih_4$ is the group of symmetries of the lattice~$\Z^2$;
taking the quotient with respect to a different lattice
leads to a homeomorphic (but not isometric) torus, thus to a possibly
different automorphism group.

\begin{figure}[h!]
\centering
\includegraphics[width=0.36\textwidth]{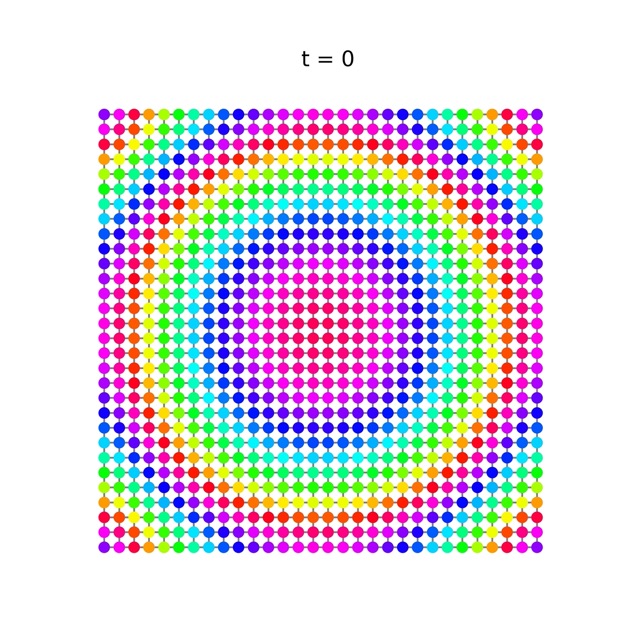}
\hspace{-25pt}
\includegraphics[width=0.36\textwidth]{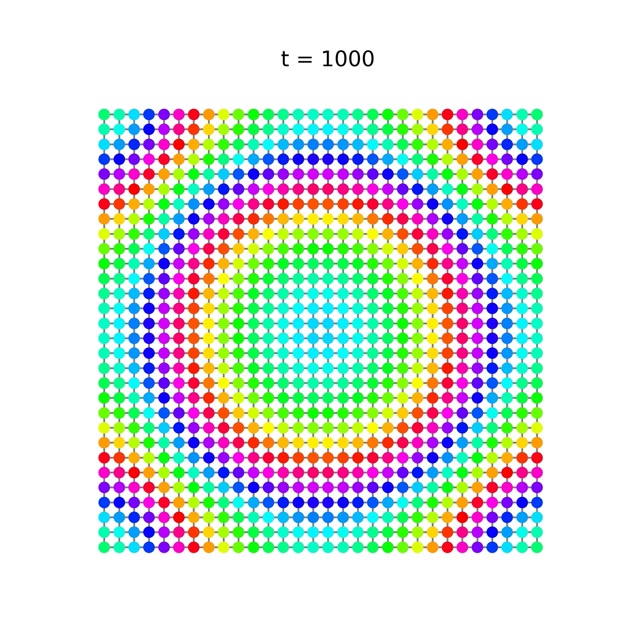}
\hspace{-25pt}
\includegraphics[width=0.36\textwidth]{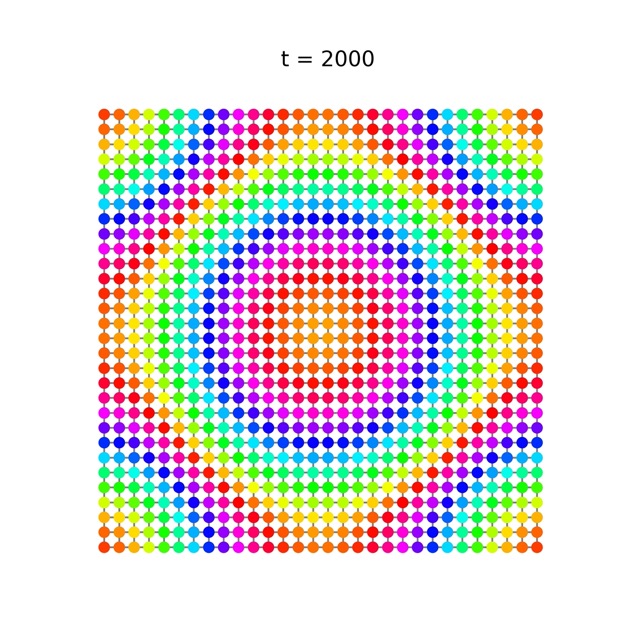}
\caption{
We simulate the evolution of an initial condition with symmetry~$(x,y)\mapsto(-y,x)$
with respect to the coupling function~$g(u_x,u_y) = \sin(u_y-u_x + 1)$.
The $2$-dimensional torus \davide{is} approximated by the graph~$C_{30}\times C_{30}$.
In the picture, vertices are arranged in a square for visual purposes:
vertices on opposite sides are understood to be connected.
The pattern is preserved over time.
}
\label{fig:square_pattern}
\end{figure}

More generally, consider the $d$-dimensional torus~$\T^d = \R^d/\Z^d$.
By Theorem~\ref{thm:iso_is_auto}, for every~$\delta$ small enough we have
\[
	\Aut(W_{\delta}) \cong \Iso(\Z^d) \ltimes \T^d
\]
where~$\Iso(\Z^d)$ is the symmetry group of the lattice~$\Z^d$.
Notice that~$\Iso(\Z^d)$ is a finite group.
As a consequence, we find some dynamically invariant subspaces.
For every~$\epsilon_1,\ldots,\epsilon_d\in \{-1,+1\}$ the subspace
given by the equation
\[
	u_{(x_1,\ldots,x_d)} = u_{(\epsilon_1 x_1,\ldots,\epsilon_d x_d)}
\]
is dynamically invariant. Indeed, notice that
for every~$\epsilon_1,\ldots,\epsilon_d\in \{-1,+1\}$ the map
\[
	(x_1,\ldots,x_d) \mapsto (\epsilon_1 x_1,\ldots,\epsilon_d x_d)
\]
is an isometry preserving the lattice~$\Z^d$, thus an isometry of~$\T^d$.
Moreover, for every direction~$(y_1,\ldots,y_d)\in \R^d$
the subspace given by the equation
\[
	u_{(x_1,\ldots,x_d)} = u_{(x_1+s y_1,\ldots,x_d + s y_d)}, \qquad s\in \R,
\]
is dynamically invariant. To show this,
notice that~$\{(s y_1,\ldots,s y_d) \mid s\in \R\}$ is a subgroup of~$\T^d$.


\subsection{Multi-twisted states}
We consider again geodesic coupling on the torus, now in the particular case of Kuramoto coupling.
Kuramoto model on a cycle graph~$C_n$ supports a family of equilibria known as~\emph{twisted states}: 
For $q\in\Z$ the $q$-twisted state is~$(2qk\pi/n)_{k=1,\ldots,n}$.
We will see that continuous analogues of the twisted states appear in geodesic graphons on tori of any dimension.

\begin{figure}[h!]
\centering
\includegraphics[width=0.49\textwidth]{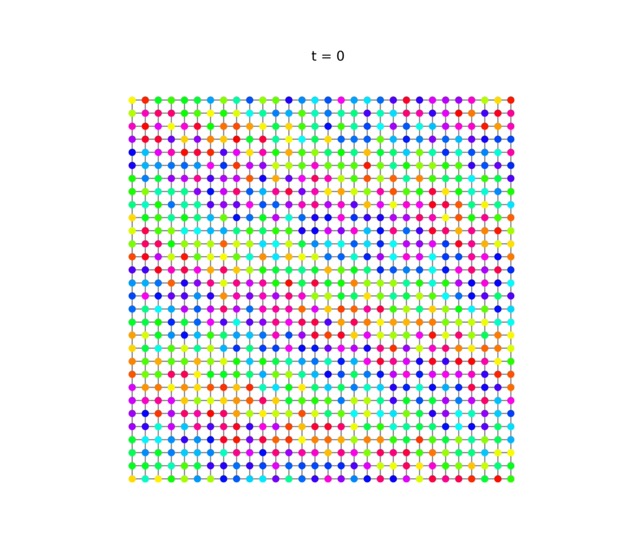}
\hspace{-30pt}
\includegraphics[width=0.49\textwidth]{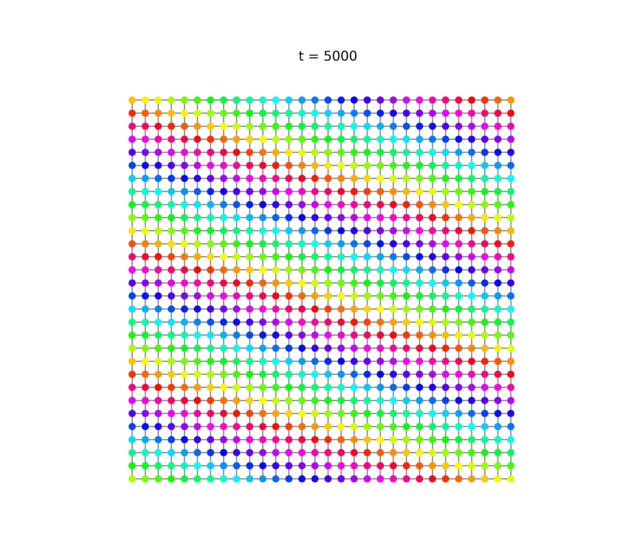}
\caption{
Convergence to a~$(1,3)$-twisted state from a perturbed configuration in~$C_{30}\times C_{30}$.
Vertices are arranged in a square for visual purposes.
}
\label{fig:twisted}
\end{figure}

For fixed~$\delta$ let~$W_\delta$ be the geodesic graphon on the $d$-dimensional
torus~$\T^d$ with respect to the
metric
\[
	\d_{\T^d}(x,y)=\max(\d_{\T}(x_1,y_1),\ldots,\d_{\T}(x_d,y_d))
\]
and the product measure~$\mu$ induced from~$\R^d$.
Consider the Kuramoto model on~$\T^d=(\R/\Z)^d$ where the phase~$\theta_x$ of $x\in \T^d$ evolves according to
\begin{equation} \label{eq:kuramoto_torus}
	\dot\theta_x = \int_{\T^d} W_\delta (x,y) \sin(\theta_y - \theta_x) \d \mu(y).
\end{equation}

\begin{proposition}
For every vector of non-zero integers~$(q_1,\ldots,q_d)$ the state
\begin{equation} \label{eq:multi-twisted}
	\theta_x= 2 \pi (q_1 x_1 + \dotsb + q_d x_d), \qquad x\in \T^d
\end{equation}
is an equilibrium of~\eqref{eq:kuramoto_torus}. We call it
$(q_1,\ldots,q_d)$-\emph{twisted state}.
\end{proposition}
\begin{proof}
{
Substituting the $(q_1,\ldots,q_d)$-twisted state into the evolution equation yields
\begin{align*}
	& \int_{\T^d} W_{\delta}(x,y)
		\sin(2\pi (q_1(y_1 - x_1) + \cdots + q_d(y_d - x_d)))\, \d \mu(y) \\
	& \qquad = \Im \p{
			\prod_{k=1}^d \int_{x_k-\delta}^{x_k+\delta} e^{2 \pi q_k i(y_k - x_k)} \, \d y_k
		} \\
	& \qquad = \Im \p{
			\prod_{k=1}^d \frac{\sin(2\pi q_k\delta)}{\pi q_k}
		} \\
	& \qquad = 0
\end{align*}
where~$\Im$ denotes the imaginary part of a complex number.
}
\end{proof}


\subsection{Spherical graphon} \label{sec:spherical_graphon}
For simplicity we restrict our analysis to the $2$-dimensional sphere~$S^2$.
In it convenient to think of~$S^2$ embedded in~$\R^3$
as the set of solutions of~$x^2+y^2+z^2 = 1$.
By Theorem~\ref{thm:iso_is_auto} we have
\[
	\Aut(W_\delta) \cong O(3).
\]
The graphon~$W_{\delta}$ with~$\delta=\pi/2$ is known as~\emph{spherical graphon},
see~\cite[Example 13.2]{lovasz2012large}.
A ``sparse'' version of this graphon, known as spherical graphop,
will be considered in Section~\ref{sec:graphop}.

As a concrete example of symmetry-induced dynamically invariant subspace,
consider the
set of functions invariant under rotation along a fixed axis.
Reduced dynamics can be represented as a graphon system parametrized by an interval,
the rotation axis.
Other examples are given by the finite subgroups of~$O(3)$, that is, the group of symmetries of the platonic solids.

\begin{figure}[h!]
\centering
\includegraphics[width=0.32\textwidth]{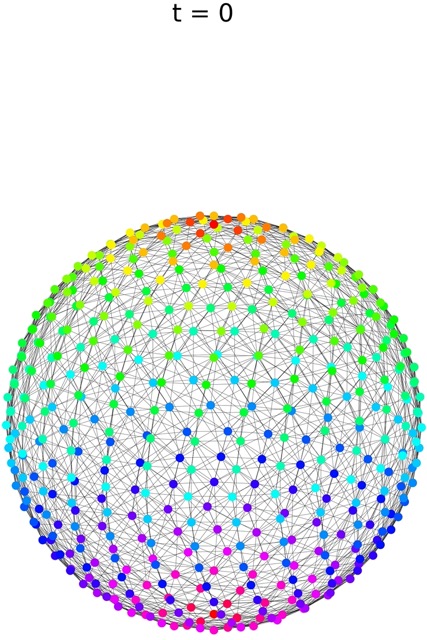}
\includegraphics[width=0.32\textwidth]{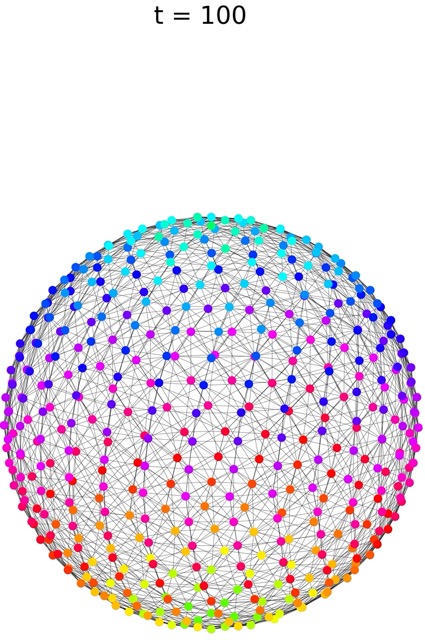}
\includegraphics[width=0.32\textwidth]{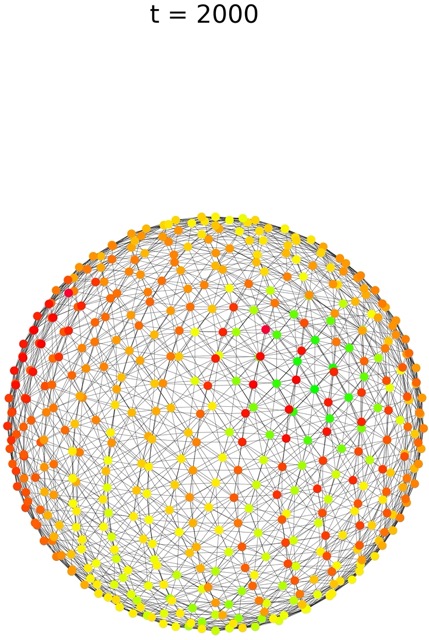}
\caption{
We consider the coupling $\sin(\theta_y-\theta_x +1)$ on a graph~$G$ with~$468$ vertices,
obtained by discretizing the sphere.
Vertices are placed on parallel circles in number proportional
to the radius.
We start with an initial condition which depends on the latitude only.
In the limit this would correspond to a solution with rotational symmetry
along the north-south axis.
Numerically we see that symmetry is approximatively preserved for some time,
but breaks for~$t$ very large.
This effect, caused by the fact that the graph~$G$ is not truly rotationally symmetric
bot only approximatively, will be discussed in Section~\ref{sec:approximate_symmetries}.
}
\label{fig:sphere}
\end{figure}


\section{Ghosts of Symmetries} \label{sec:approximate_symmetries}
\label{sec:approximated_symmetries}

For a moment we return to graphs and graphons as combinatorial objects. If we consider a convergent family of graphs, then the automorphisms of each element may be starkly different from the automorphisms of the limit graphon; cf.~\cite{lovasz2015automorphism}.
For example, with high probability a large \ER random graph
has no symmetries~\cite[Corollary 2.3.3]{godsil2001algebraic}
while the constant limit graphon has many (Section~\ref{sec:constant_graphon}).

Now turning to dynamics, a key reason to analyze dynamics on graphons is that they approximate the dynamics of large but finite graphs, which are much less analytically tractable. 
One can make this rigorous by proving explicit statements of how close the dynamics on finite graphs are to the dynamics on the limit object (at least for some finite time);
see for example~\cite{Chiba2019, Chiba2016, Medvedev2017, Medvedev2013a, kaliuzhnyi2018mean}.

Therefore, there are families of converging graphs supporting very few synchrony patterns
(due to the lack of symmetries)
but the dynamics on the limits have a lot (due to the many symmetries). 
But since the dynamics of finite and limit system are close, we know that even the finite, non-symmetric dynamical system can show finite-time synchronous behavior due to the dynamics of the limit object.
So for finite graph dynamical systems we can see \emph{ghosts of symmetries} of the limit graphon dynamical system: We observed this effect for \ER random graphs and for spherical graphs in Figure~\ref{fig:sphere} and is related to the metastable dynamics observed
in~\cite[Figure 5]{medvedev2015stability}.
The important consequence is we should not be surprised if graph dynamical systems on large but highly non-symmetric graphs show finite-time synchronized dynamics---understanding the symmetries of the limit object is the key to this insight.
The goal of this section is formalizing this observation.

How well graph dynamics approximate the graphon limit dynamics
depends on the choice of~$f$ and~$g$.
For concreteness we focus on Kuramoto dynamics
\begin{equation} \label{eq:approximate:kuramoto}
	\dot u_x = \int_{J} W(x,y) \sin(u_y - u_x) \ \d \mu (y)
\end{equation}
in line with previous investigations.
In~\cite{Medvedev2013a} the author assumes convergence of~$(J,W^{(n)})$
to~$(J,W)$ in~$L^1(J^2)$. Although this is convenient for the convergence of dynamics,
it excludes several convergent graph sequences that are of interest here,
like \ER random graphs.
In order to include these cases, we will prove a stronger statement.

Indeed, we will assume convergence in the weaker norm~$\enorm_{\infty\to 1}$,
which is equivalent to convergence in terms of homomorphism densities
up to vertex relabeling~\cite[Theorem 11.59]{lovasz2012large}.
Indeed we show that the norm~$\enorm_{\infty\to 1}$,
which is strictly weaker than the $L^1$-norm,
is actually sufficient to guarantee convergence of dynamics in the case
of sine coupling.
In the following theorem we prove~$\enorm_{\infty\to 1}$-continuity
with respect to the graphon and~$\enorm_{1}$-continuity with respect to the initial
condition:

\begin{theorem} \label{thm:continuity_of_dynamics}
Consider two graphons~$W,U$ on the same index space~$J$
and let~$\flow^W,\flow^U$ denote the flows induced by
Kuramoto coupling~\eqref{eq:approximate:kuramoto} respectively.
Let~$u(0)$, $v(0)\in L^1(J)$. Define~$u(t) = \flow_t^W(u)$ and~$v(t) = \flow_t^U(v)$.
Then for every~$t\in\R$
\[
	\norm{u(t) - v(t)}_1 \leq
		\p{
			\norm{u(0)-v(0)}_1 + 2t \norm{W - U}_{\infty\to 1}
		}
	e^{2t}.
\]
\end{theorem}
\begin{proof}
For every~$s\in \R$ and every~$x\in J$
we have
\begin{subequations}
\begin{align}
	\dot u_x (s) &- \dot v_x (s) = \label{eq:kur:0} \\
		& \int \p{W(x,y) - U(x,y)} \sin(u_y(s) - u_x(s)) \d y \label{eq:kur:1} \\
		& - \int U(x,y) \p{\sin(v_y(s) - v_x(s)) - \sin(u_y(s) - u_x(s))} \d y \label{eq:kur:2}.
\end{align}
\end{subequations}
From the addition formula for sine it follows that~\eqref{eq:kur:1} is bounded by 
\begin{align*}
	& \norm{W - U}_{\infty\to 1} \norm{\sin(u(s))}_\infty \abs{\cos(u_x(s))} \\
	& \quad + \norm{W - U}_{\infty\to 1}  \norm{\cos(u(s))}_\infty \abs{\sin(u_x(s))} \\
	& \leq 2 \norm{W - U}_{\infty\to 1}.
\end{align*}
On the other hand~\eqref{eq:kur:2} is bounded by
\[
		\abs{v_y(s) - u_y(s)} + \abs{v_x(s) - u_x(s)}.
\]
Therefore, integrating~\eqref{eq:kur:0},~\eqref{eq:kur:1}, and~\eqref{eq:kur:2}
on~$J$ gives
\[
	\int \abs{\dot u_x (s) - \dot v_x (s)} \d x \leq
		2 \norm{W - U}_{\infty\to 1} + 2 \norm{u\davide{}-v}_1.
\]
Integrating with respect to time gives
\[
	\norm{u(t) - v(t)}_1 \leq \norm{u(0)-v(0)}_1
		+ 2 t \norm{W - U}_{\infty\to 1} + \int_0^t 2 \norm{u(s)-v(s)}_1 \d s.
\]
Gr\"onwall's Lemma completes the proof.
\end{proof}

Let~$(J, W^{(n)})_n$ be a sequence of graphons which converges to~$(J,W)$ in the infinity to one norm.
Theorem~\ref{thm:continuity_of_dynamics} implies that, for every fixed~$t$, the flow
\[
	\flow^{W^{(n)}}_t: L^1(J) \to L^1(J)
\]
converges uniformly, as a map, to
\[
	\flow^{W}_t: L^1(J) \to L^1(J).
\]
Moreover, convergence is uniform in time
if~$t$ is restricted to any compact interval~$[-T,+T]$.
In particular:

\begin{corollary}
Let~$(J, W^{(n)})_n$ be a sequence of graphons which converges to~$(J,W)$
in the infinity to one norm. Let~$u^{(n)} \in L^1(J)$
be a sequence of initial conditions converging to~$u\in L^1(J)$ in the $L^1$-norm.
Then for every~$t\in \R$
\[
	\lim_{n\to \infty} \norm{u^{(n)}(t)-u(t)}_1 = 0
\]
holds. Convergence is uniform in~$t$ on any compact interval~$[-T,+T]$.
\end{corollary}

Theorem~\ref{thm:continuity_of_dynamics} explains
the almost symmetric dynamics of large systems whose limit is symmetric:

\begin{corollary} [The Ghost of Symmetries] \label{cor:ghosts}
Let~$(J,W)$ be a graphon with automorphism~$\varphi$
and let~$u(t) \in \D(J,W)$ be a trajectory satisfying
$
	\varphi^* (u(0)) = u(0).
$
Let~$(J, W^{(n)})_n$ be a sequence of graphons
and~$u^{(n)}(t) \in \D(J,W^{(n)})$ a sequence of trajectories
such that
\[
	\lim_{n\to \infty} \norm{W^{(n)} - W}_{\infty\to 1} = 0,
	\qquad
	\lim_{n\to \infty} \norm{u^{(n)}(0) - u(0)}_1 = 0.
\]
Then for every~$t\in \R$
\[
	\lim_{n\to \infty} \norm{\varphi^* (u^{(n)}(t)) - u^{(n)}(t)}_1 = 0.
\]
Moreover, convergence is uniform in~$t$ on any compact interval~$[-T,+T]$.
\end{corollary}
\begin{proof}
By triangular inequality
\[
	\norm{\varphi^* (u^{(n)}(t)) - u^{(n)}(t)}_1
	\leq \norm{\varphi^* (u^{(n)}(t)) - \varphi^* (u (t))}_1
		+ \norm{\varphi^* (u(t)) - u^{(n)} (t)}_1.
\]
Consider the two terms in the right hand side.
Since~$\varphi^*$ is an isometry of~$L^1$
then the first one is equal to~$\norm{u^{(n)}(t) - u (t)}_1$.
Since~$\varphi^*$ is a symmetry of~$\D(J,W)$
then the second term is equal to~$\norm{u^{(n)}(t) - u (t)}_1$.
Therefore
\[
	\norm{\varphi^* (u^{(n)}(t)) - u^{(n)}(t)}_1 \leq 2 \norm{u^{(n)}(t) - u (t)}_1.
\]
Theorem~\ref{thm:continuity_of_dynamics} implies
\begin{equation} \label{eq:approximate_estimate}
	\norm{\varphi^* (u^{(n)}(t)) - u^{(n)}(t)}_1 \leq 2
	\p{
			\norm{u^{(n)}(0)-u(0)}_1 + 2t \norm{W^{(n)} - W}_{\infty\to 1}
		}
	e^{2t}.
\end{equation}
Taking the limit concludes the proof.
\end{proof}

Equation~\eqref{eq:approximate_estimate} shows how much the trajectory of an approximatively
symmetric system can deviate from being symmetric for any fixed time.
In particular it can be used to estimate the minimal time needed for a
given large deviation from symmetry.

Note that an additional step is needed to apply Corollary~\ref{cor:ghosts} to simulations.
Indeed, in simulations we approximated the limit graphon dynamical system~$\D(J,W)$, defined
on some continuum index space~$J$,
with a graph dynamical system~$\D(G^{(n)})$
on a different index space, namely the discrete set~$\{1,\ldots,n\}$.
In order to apply Corollary~\ref{cor:ghosts} one has to first embed
every finite system in the same index space~$J$ where the limit is defined.
There are many ways of doing so; if~$J=I$ the canonical embedding is one such way.
If the graph sequence~$G^{(n)}$ converges to the graphon~$(J,W)$
in terms of homomorphism densities then it is always possible to represent~$G^{(n)}$
as a graphon~$W^{(n)}$ on~$J$ in a way that makes~$\norm{W^{(n)}-W}_{\infty\to 1}$ 
converge to zero.

\davide{
Notice that the proof of Corollary~\ref{cor:ghosts} does not require~$\varphi$
to be invertible. In other words, one may also observe ghosts of generalized symmetries.
}


\section{Symmetries of Mean-Field Graphon Dynamical Systems} \label{sec:mean-field-graphon}

In this section we consider generalizations of graphon dynamical systems
in which states, instead of being real-valued, are measure-valued.
So far, the state of $x\in J$ was given by some number~$u_x \in \R$. By contrast, mean-field transport equations describe not the evolution of points but rather general probability measures. 
Mean-field equations on convergent families of graphs have been introduced in~\cite{kaliuzhnyi2018mean}; we will refer to these systems as \emph{mean-field graphon dynamical systems} and make this precise below.
We now generalize graphon-induced symmetries to such mean-field graphon systems.
The main result of this section is to exploit symmetry arguments to show that mean-field graphon system contains a range of simpler systems that have been analyzed independent of one another literature.
For concreteness, we follow~\cite{kaliuzhnyi2018mean} and restrict our analysis to the case of phase oscillators and assume~$J=I$.

\newcommand		{\M} 	{\mathcal M}
\newcommand		{\MI}	{\M^\I}
\newcommand		{\TI}	{\T^\I}

We first take a step back and recall the notion of mean-field dynamics for a globally coupled Kuramoto oscillator network~\eqref{eq:graph:kuramoto}---we refer to this as the corresponding \emph{mean-field dynamical system}. 
Let~$\T = \R/\Z$ be the $1$-dimensional torus and let~$\M$ be the set of Borel probabilities on~$\T$. Rather than tracking the state of each oscillators, the mean-field dynamical system
describes the evolution of a measure~$\mu \in \M$ over time, given by
\begin{equation} \label{eq:mean-field-push}
	\mu(t) = \Phi(t) \# \mu_0
\end{equation}
where~$\Phi$ is the flow on~$\T$ induced by the differential equation
\begin{equation} \label{eq:mean-field-diff}
	\dot u(t) = \int_\T \sin(v-u(t))\ \d\mu(t)(v)
\end{equation}
and~$\#$ denotes the push-forward of measures.
Together, equations \eqref{eq:mean-field-push}~and~\eqref{eq:mean-field-diff} represent the mean-field dynamics of~\eqref{eq:graph:kuramoto} on complete graphs as~$n\to \infty$.

While in the mean-field dynamical system we only track the state distribution of all oscillators as a whole, in the graphon dynamical systems we consider the dynamics on the corresponding graph limit (the constant graphon on $J=I$).
Indeed, the mean-field system can be obtained from the graphon dynamical system by forgetting the labels:
To make this statement precise, let~$\lambda$ denote the Lebesgue measure on~$I$, let~$W=1$ and consider a solution~$t\mapsto u(t)$ of the graphon system
\[
	\dot u_x(t) = \int_I W(x,y) \sin(u_y(t) - u_x(t)) \ \d \lambda (y).
\]
Then the measure~$\mu(t) = u(t) \# \lambda$ is a solution of the mean-field system~\eqref{eq:mean-field-push}:
\[
	\int_I \sin(u_y(t)-u_x(t)) \ \d \lambda (y) =
	\int_{\T} \sin(v-u_x(t)) \ \d \mu(t)(v).
\]

The mean-field graphon dynamical system now combines the two aspects: It describes the evolution of a family of measures indexed by~$I$, a measure-valued measurable function~$I\to \M$.
We can interpret the latter in two ways:
First, as a graphon system with non-deterministic states;
second, as infinitely many mean-field systems coupled through a graphon.

Let~$\MI$ denote the set of maps~$I\to \M$ such that the preimages of open
sets are measurable. This space is endowed with a metric~\cite{kaliuzhnyi2018mean}.
Fix a continuous map~$\mu: \R \to \MI$, an index~$x\in I$
and a graphon~$W$.
For every~$t_0\in \R$ and~$u_0\in\T$
there exist a unique solution~$u_x: \R\to \T$ to the following initial value
problem~\cite[Lemma 2.2]{kaliuzhnyi2018mean}:
\begin{equation} \label{eq:mfg_flow}
\begin{cases}
	\dot u_x(t) = \int_{I} W(x,y) \int_{\T} \sin(v-u_x(t)) \ \d \mu(y,t)(v) \ \d y \\
	u_x(t_0) = u_0.
\end{cases}
\end{equation}
Let~$\Phi(\mu,x,t,\c)$ denote the flow induced by the initial value problem.
Notice that we have a flow for each choice of~$x$.

\begin{definition}
The \emph{mean-field graphon dynamical system} is given by
\begin{subequations}\label{eq:mfg}
\begin{align}
	\mu(x,t) &= \Phi(\mu,x,t,\c)\# \mu(x,0)  \label{eq:mfg_push} \\
	\mu(x,0) &= \mu_0(x) \label{eq:mfg_init}
\end{align}
\end{subequations}
where~$\Phi(\mu,x,t,\c):\T\to \T$ is the flow induced by~\eqref{eq:mfg_flow}
and~$\mu_0\in \MI$. These equations are required to hold for every~$t\in \R$
and almost every~$x \in I$.
\end{definition}

Existence and uniqueness of the mean-field graphon system
are proved in~\cite{kaliuzhnyi2018mean} for~$t\in [0,T]$
and can be extended to~$t\in \R$ as in Lemma~\ref{lem:existence_uniqueness}.

Graphon-induced symmetries act on the index space~$I$ and not on the state space.
As a consequence, the graphon-induced invariant subspaces
in the graphon system~$I\to \R$ and in the mean-field graphon system~$I\to \M$
are essentially the same, in the sense that they are given by the same equations.
In the following theorem we prove that indeed graphon-induced symmetries
extend to the mean-field graphon system.

\begin{theorem} \label{prop:sym_mp}
Let~$\gamma:\I\to \I$ be a measure-preserving transformation satisfying~$W^\gamma=W$.
Then~$\gamma$ acts on~$\MI$ by sending a measured-valued function~$\mu: x\mapsto \mu(x)$
to the measure-valued function~$\mu\circ\gamma: x\mapsto \mu(\gamma(x))$.
The action is a symmetry
of the mean-field graphon system.
\end{theorem}

\begin{proof}
Let~$\gamma$ be as in the statement.
The elements of~$\MI$ are the maps~$I\to \M$ such that the preimages of open sets are
measurable~\cite{kaliuzhnyi2018mean}.
Fix~$\mu\in\MI$ and a measurable set~$\mathcal A \subseteq \M$.
Then~$\p{\mu\circ\gamma}^{-1} \mathcal A$ is equal to
the set~$\gamma^{-1} (\mu^{-1} \mathcal A)$.
This set is measurable since~$\mu^{-1} \mathcal A$ is a measurable set and~$\gamma:\I\to\I$
is a measurable function.
This proves that~$\mu\circ\gamma\in \MI$.

We now prove that~$\gamma$ is a symmetry of the dynamical system.
Let~$\mu(x,t)$ be a solution of~\eqref{eq:mfg} with initial value~$\mu_0(x)$.
We need to show that~$\mu(\gamma(x),t)$ is a solution
with initial value~$\mu_0(\gamma(x))$. The initial value condition is trivial.
It remains to prove that~\eqref{eq:mfg_push} is satisfied if~$\mu(x,t)$, $\mu$ and $\mu(x,0)$
are replaced by~$\mu(\gamma(x),t)$, $\mu\circ\gamma$ and $\mu(\gamma(x),0)$
respectively.
By definition
\[ \mu(x,t) = \Phi (\mu, x, t, \c) \# \mu(x,0) \]
holds for almost every $x$. Since $\gamma$ is measure-preserving then
\[ \mu(\gamma(x),t) = \Phi (\mu, \gamma(x), t, \c) \# \mu(\gamma(x),0) \]
holds for almost every $x$ as well. Since our goal is to prove that
\[ \mu(\gamma(x),t) = \Phi (\mu\circ \gamma, x, t, \c) \# \mu(\gamma(x),0), \]
it remains to be shown that
$\Phi (\mu, \gamma(x), t, \c)$ is equal to
$\Phi (\mu\circ \gamma, x, t, \c)$.
We compare the equations defining these flows.
Note that $\Phi (\mu\circ\gamma, x, t, \c)$ is given by
$$
\dot u(t) = \int_\I W(x,y) \int_{\T} D(v-u(t)) \ \d \mu(\gamma(y),t)(v) \ \d y
$$
while $\Phi (\mu, \gamma(x), t, \c)$ is given by
$$
\dot u(t) = \int_\I W(\gamma(x),y) \int_{\T} D(v-u(t))\ \d \mu(y,t)(v) \ \d y.
$$
The latter can be obtained from the former
by replacing $W(x,y)$ with $W(\gamma(x),\gamma(y))$ and
then~$\gamma(y)$ with~$y$. Since~$W^\gamma=W$
and~$\gamma$ is measure-preserving, then none
of these operations change the right hand side.
\end{proof}

We now return to constant coupling~$W=1$. Inspired by Section~\ref{sec:constant_graphon} we look at the cluster space
\[
	\{\mu\in \MI \mid x\mapsto \mu(x) \text{ is constant} \} \cong \M.
\]
This set is dynamically invariant and the mean-field graphon dynamics restricted to
the set is the same as the mean-field dynamics
(see $F$ in Figure~\ref{fig:mean_field_map}).
Indeed, since~$W=1$ and~$y\mapsto \mu(y)$
is constant we have
\[
	\int_{I} W(x,y) \int_{\T} \sin(v-u) \ \d \mu(y,t)(v) \ \d y
	= \int_{\T} \sin(v-u) \ \d \mu(y,t)(v).
\]
Mean-field graphon dynamics represents the evolution of infinitely many coupled
mean-field populations. When all the populations start with the same initial condition---that is, they are synchronized as probability measures---it reduces to the evolution of just one population
(see~$A$ in Figure~\ref{fig:mean_field_map}).

\begin{figure}
\centering
\includegraphics[width=\textwidth]{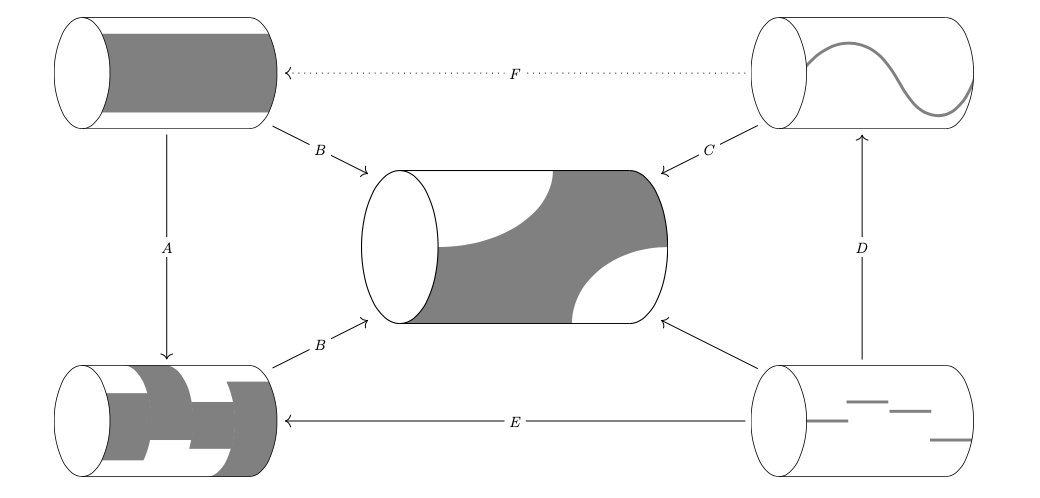}
\caption{
We sketch the inclusions of different network dynamical systems.
We use cylinders~$I\times \T$ to represent the state space of the dynamical units in different
limits. On the top-left, the mean-field limit,
the interval~$I$ can be identified to a point (one population)
whose state is described by one probability measure on~$\T$ (whose support is
represented in gray).
The bottom-left generalizes this mean-field to more than one population.
If the probability measures of the multi-population mean-field are Dirac,
we recover the finite system, bottom-right.
The top-right represents the graphon system.
The most general mean-field graphon system, in the center, contains all the others
as particular cases.
}
\label{fig:mean_field_map}
\end{figure}

Now consider a graphon with a block structure~$W_G = \sum_{j,k} G_{j,k} \1_{I_j\times I_k}$
as in Section~\ref{sec:block_structure}
and the cluster space
\[
	\bigcap_{k=1}^n \{\mu\in \MI \mid \mu: I_k\to \M \text{ is constant} \} \cong \M^n
\]
associated to the partition~$I=\bigcup_{k=1}^n I_k$.
Each block evolves as a single mean-field system.
Cluster dynamics represents the evolution of $n$~measures
coupled by the graph~$G$. This system is known as multi-population mean-field limit
and has been analyzed in~\cite{bick2022multi}.

We have seen how mean-field and multi-population mean field are
dynamically invariant subspaces of the mean-field graphon system
if the underlying graphon has certain symmetries.
Now we see how the graphon system can also be seen as a dynamically invariant subset
of the mean-field graphon system (see~$C$ in Figure~\ref{fig:mean_field_map}).
Intuitively, the graphon system is deterministic and the states can be interpreted 
as Dirac measures. If~$\mu(x) = \delta_{u(x)}$ is a Dirac measure for every~$x$
then the push-forward~\eqref{eq:mfg_push} is a Dirac measure for every~$x$ and~$t$
and the flow~$\Phi(\delta_{u}, x, t, \cdot)$ of~$x$ reduces to the graphon evolution equation:
\[
	\int_{I} W(x,y) \int_{\T} \sin(v-u_x) \ \d \mu(y,t)(v) \ \d y
	= \int_{I} W(x,y) \sin(u_y-u_x) \ \d y.
\]
Figure~\ref{fig:mean_field_map} summarizes the inclusions of invariant subspaces of mean-field graphon dynamical systems.


\section{Symmetries of Graphops and Graphop Dynamical Systems} 
\label{sec:graphop}

Finally, we consider the symmetries and their implications for dynamical systems on graphops; see for example~\cite{Kuehn2020b}. While graphons are natural limit objects for converging sequences of dense graphs, there are other notions of graph convergence such as Benjamini--Schramm convergence for sparse graph sequences~\cite{benjamini2011recurrence}. Graphops interpret graph sequences and their limits as operators and provide a unifying framework for different notions of convergence~\cite{backhausz2022action}. We now consider a generalization of the graph dynamical system~\eqref{eq:main:graph} to graphops---we refer to them as \emph{graphop dynamical systems}---and analyze their symmetry properties.

\subsection{Graphop dynamical systems}
Fix a probability space~$J=(\Omega, \mathcal A, \mu)$. Recall that a graphon dynamical system~\eqref{eq:main:graphon} evolves according to
\[
	\dot u_x = f\p{u_x, \int_{J} W(x,y) g(u_x, u_y) \d \mu (y)}.
\]
We can interpret~$\nu_x = W(x,\cdot) \d \mu$ as a measure that has a density~$W(x,\cdot)$ with respect to~$\mu$.
What now determines the dynamics is the joint effect of graphon and measure rather than their individual effect.
Intuitively, graphops generalize graphons by removing the hypothesis of having a density.

A \emph{graphop} is a self-adjoint, positive-preserving
bounded operator $L^\infty (J) \to L^1 (J)$. Every graphop can be represented by
a symmetric measure~$\nu$
on the product space~$(\Omega\times \Omega, \mathcal A \times \mathcal A)$~\cite[Theorem 6.3]{backhausz2022action}.
By using the disintegration theorem one obtains a family of measures~$\{\nu_x\}_{x\in \Omega}$,
called \emph{fiber measures}, that uniquely
represent the graphon~\cite[Remark 6.4]{backhausz2022action}.
We call \emph{graphop dynamical system} the dynamical system on~$L^1(J)$ induced by
\[
	\dot u_x = f\p{u_x, \int_{\Omega} g(u_x, u_y) \d \nu_x (y)}.
\]

We propose the following definition of graphop automorphism:

\davide{
\begin{definition} \label{def:graphop:automorphism}
Let~$(\nu_x)_{x\in \Omega}$ be a graphop on~$J=(\Omega, \mathcal A, \mu)$.
An \emph{automorphism} of the graphop
is a measurable bijection~$\varphi: (\Omega, \mathcal A) \to (\Omega, \mathcal A)$
such that~$\varphi^*: L^1(J) \to L^1(J)$
and satisfying~$\varphi \# \nu_x = \nu_{\varphi(x)}$ for every~$x$.
\end{definition}
}

The identity~$\varphi \# \nu_x = \nu_{\varphi(x)}$ is shorthand of two things:
First, the map~$\varphi: (\Omega, \nu_x) \to (\Omega, \nu_{\varphi(x)})$ is measurable and, second, it is measure preserving. Notice that, as oppose to graphon automorphisms,
we do not require~$\varphi: (\Omega, \mu) \to (\Omega, \mu)$ to be measure-preserving.
A consequence of this fact is that graphop symmetries
are symmetries of graphop dynamical systems,
but in general they are not isometries.

\begin{lemma} \label{lem:graphop_symmetry}
Let~$(\nu_x)_{x\in \Omega}$ be a graphop on~$J=(\Omega, \mathcal A, \mu)$.
Let~$\varphi$ be a graphop automorphism. Then
\[
	\varphi^*: L^1(J) \to L^1(J), \quad u \mapsto u\circ \varphi
\]
is a symmetry of the graphop dynamical system. We call~$\varphi^*$ a
\emph{graphop-induced} symmetry.
\end{lemma}
\begin{proof}
Let~$u \in L^1(J)$.
By definition of graphop automorphism and the change of variable formula we have
\begin{align*}
		 & f\p{ u_{\varphi(x)}, \int_{\Omega} g(u_{\varphi(x)},u_y) \d \nu_{\varphi(x)}(y)} \\
		 & = f\p{ u_{\varphi(x)}, \int_{\Omega} g(u_{\varphi(x)},u_y) \d (\varphi \# \nu_x) (y)} \\
		 & = f\p{ u_{\varphi(x)}, \int_{\Omega} g(u_{\varphi(x)},u_{\varphi(y)}) \d \nu_{x}(y)}.
\end{align*}
This shows that~$\varphi^*$ maps solutions to solutions.
\end{proof}

Every graphon is a graphop. In the following lemma we show that graphon-induced symmetries
are graphop-induced symmetries, although we will see
that the converse is false in general.

\begin{lemma}
Let~$W$ be a graphon on~$J$ and~$\{\nu_x = W(x,\cdot) \d \mu\}_{x\in J}$ the associated graphop.
Every graphon automorphism is a graphop automorphism.
A graphop automorphism~$\varphi$ is a graphon automorphism if and only if
it preserves~$\mu$. In particular, the graphop-induced symmetries are exactly
the graphon-induced symmetries that are isometries of~$L^1(J)$.
\end{lemma}
\begin{proof}
Let~$\varphi: J\to J$ be an invertible measurable map. Fix~$x\in J$.
Notice that
\[
	W(\varphi(x),\varphi(y))=W(x,y)
\]
holds for almost every~$y\in J$ if and only if
\[
	W(\varphi(x),y)=W(x,\varphi^{-1}(y))
\]
holds for almost every~$y \in J$. Moreover,
notice that~$\varphi$ is an automorphism of~$(\nu_x)_x$
if and only if for every~$x$ and every measurable function~$p$ we have
\[
	\int_J p(\varphi(y)) W(x,y) \d \mu(y) = \int_J p(y) W(\varphi(x),y) \d \mu(y)
\]
which is the same as
\begin{equation} \label{eq:compare}
	\int_J p(y) W(x,\varphi^{-1} (y)) \d (\varphi \# \mu) (y)
		= \int_J p(y) W(\varphi(x),y) \d \mu(y).
\end{equation}
Therefore, if~$\varphi$ is a graphon automorphism then~$W(x,\varphi^{-1} (y))=W(\varphi(x),y)$
for every~$x$ and almost every~$y$. This implies~$\varphi \# \mu = \mu$ and,
by definition, that~$\varphi$ is a graphop automorphism.

On the other hand, if a graphop automorphism~$\varphi$ preserves~$\mu$, then for every~$x$
and every measurable function~\davide{$p$}
\[
	\int_J p(y) W(x,\varphi^{-1} (y)) \d \mu (y)
		= \int_J p(y) W(\varphi(x),y) \d \mu(y).
\]
This implies~$W(x,y) = W^\varphi (x,y)$ for every~$x$ and almost every~$y$.
Therefore every measure-preserving graphop automorphism is a graphon
automorphism.

The last statement follows form the fact that a measurable map~$\varphi: J\to J$
is measure-preserving if and only if~$\varphi^* : L^1(J) \to L^1(J)$ is an isometry.
\end{proof}

\davide{In the context of graphon dynamical systems we have introduced the notion
of generalized graphon-induced symmetry (Definition~\ref{def:graphon_generalized_symmetry}).
These are not graphon automorphisms, as they are map between two graphons and are not
necessarily invertible. One might do the same for graphop dynamical systems.}

\subsection{Spherical graphop and other examples} 
\label{sec:spherical_graphop}

We illustrate the results with two concrete examples. The first example highlights that graphop-induced symmetries can be richer than graphon symmetries: While graphon symmetries only take the structure of the kernel into account, graphop-induced symmetries combine kernel and index space that both determine the network dynamics. In the second example, we compute the symmetry group of the spherical graphop, a generalization of the spherical graphon in Section~\ref{sec:spherical_graphon}.

\subsubsection*{Finite graphop}
Consider a graphon~$W$ with four vertices~$a,b,c,d$ with
probability measure~$\mu=(2/9,1/9,4/9,2/9)$ and two edges~$ab$ and~$cd$
with weights~$W(a,b)=1$ and~$W(c,d)=1/2$:
\begin{center}
\begin{tikzcd}
a \arrow[dash]{d}{1} & c \arrow[dash]{d}{1/2} \\
b & d
\end{tikzcd}
\qquad \quad
$
\begin{cases}
\mu(\{a\}) = 2/9, \\
\mu(\{b\}) = 1/9, \\
\mu(\{c\}) = 4/9, \\
\mu(\{d\}) = 2/9.
\end{cases}
$
\end{center}

The only non-identical measure preserving transformation
is the permutation interchanging $a$~with~$d$,
which does not preserve adjacency.
Therefore, the only graphon automorphism is the identity map.
Direct computation shows that the fiber measures of the associated graphop are
\[
	\nu_a = \p{0, \frac{1}{9}, 0, 0}, \,
	\nu_b = \p{\frac{2}{9}, 0, 0, 0}, \,
	\nu_c = \p{0, 0, 0, \frac{1}{9}}, \,
	\nu_d = \p{0, 0, \frac{2}{9}, 0}.
\]
Therefore, the permutation interchanging $a$~with~$c$ and $b$~with~$d$
is a graphop automorphism. Indeed, it is a symmetry of the associated dynamical system
\begin{align*}
\begin{cases}
\dot u_a = f\p{u_a, \frac{1}{9} g(u_a, u_b)} \\
\dot u_b = f\p{u_b, \frac{2}{9} g(u_b, u_a)} \\
\dot u_c = f\p{u_c, \frac{1}{9} g(u_c, u_d)} \\
\dot u_d = f\p{u_d, \frac{2}{9} g(u_d, u_c)}.
\end{cases}
\end{align*}

\subsubsection*{Spherical graphop}
The spherical graphop is a typical example of graph limit
of intermediate density, see~\cite[Figure 4]{backhausz2022action}
and~\cite[Example 5.4]{gkogkas2022graphop}.
The spherical graphop is a sparse version of the spherical graphon
we considered in Section~\ref{sec:spherical_graphon}.
As for the spherical graphon, we compute the automorphism group
and discuss the consequence for dynamics.

Let~$J$ be the unit sphere~$S^2 = \{x\in \R^3 \mid x^{\textsf{T}} x = 1\}$
endowed with uniform probability measure.
For every~$x\in S^2$ let~$\nu_x$ be the uniform measure
on the great circle~$\{y\in S^2 \mid y^{\textsf{T}} x = 0 \}$ orthogonal to~$x$.
The \emph{spherical graphop} is given by the collection of fiber measures~$(\nu_x)_{x\in S^2}$.

As one may expect, the spherical graphon and spherical graphop have the same automorphism group. But the proofs are quite different:
While the graphon case involved geometric measure theory, for the graphop case we will use projective geometry and \davide{} arguments in~\cite{Kohan}.
If automorphisms are assumed to be smooth (which \davide{is} a consequence of the following result),
an alternative proof using differential geometry can be given.

\begin{proposition} \label{prop:spherical_graphop}
The automorphism group of the spherical graphop is the isometry group of the sphere~$O(3)$.
\end{proposition}

\begin{proof}
Every orthogonal transformation is measurable and preserves
adjacency, thus it is a graphop automorphism.

Conversely, for an automorphism~$\varphi$ of the spherical graphop we have to show that $\varphi \in O(3)$.
The map~$\varphi$ is a bijective map~$S^2\to S^2$ which maps orthogonal vectors to orthogonal vectors. 
In particular,~$\varphi$ preserves antipodal pairs.
The quotient of~$S^2$ with respect to the antipodal pair relation is
canonically identified with the projective plane, the quotient of~$\R^3$
with respect to the linear dependence relation. Therefore~$\varphi$
induces a bijection of the projective plane with itself.

Since~$\varphi$ preserves orthogonality, then~$\varphi$ maps great circles to great circles.
The great circles in~$S^2$ corresponds to the projective lines in the projective plane.
Therefore~$\varphi$ induces a bijection of the projective plane preserving
collinearity. The Fundamental Theorem of Projective Geometry states that any such map
is given by an invertible linear transformation~$T$ of~$\R^3$.
The transformation~$T$ is uniquely determined up to scalar multiplication.

Any linear transformation preserving orthogonality is a scalar multiple of
an orthogonal transformation.
Therefore~$\varphi = \lambda T$ for some~$\lambda \neq 0$ and some orthogonal transformation~$T$.
Since~$\varphi$ maps~$S^2$ to itself then~$\abs{\lambda}=1$.
We conclude that~$\varphi \in O(3)$.
\end{proof}

Dynamics on the spherical graphop, considered
in~\cite[Example 5.4]{gkogkas2022graphop}, has a group of symmetries~$O(3)$.
For example, the space of functions invariant under rotation along a fixed axis
is dynamically invariant.
Reduced dynamics can be represented as a graphop system parametrized the invariant diameter.


\renewcommand\refname{References}
\bibliographystyle{siam}
\bibliography{refs}

\end{document}